\documentclass[11pt,twoside]{article}

\usepackage[left=3cm, right=3.5cm, top=3cm]{geometry}
\usepackage{amssymb}
\usepackage{amsmath}
\usepackage{amsthm}
\usepackage{mathtools}
\usepackage{listings}
\usepackage{epsfig}
\usepackage{latexsym}
\usepackage{mathtools}
\usepackage{booktabs}
\usepackage{graphicx}
\usepackage{xcolor}
\usepackage{epsfig}
\usepackage{url}
\usepackage{amsfonts}

\usepackage{hyperref}
\usepackage{tipa}
\usepackage{dsfont}
\usepackage[mathscr]{euscript}
\usepackage{enumitem}
\usepackage{bm}
\usepackage[backend=bibtex,style=numeric,citestyle=numeric,giveninits=true,sortcites=true,maxbibnames=2]{biblatex}
\usepackage{dsfont}
\usepackage{subcaption}
\usepackage{appendix}
\usepackage{float}

\usepackage{filecontents}

\immediate\write18{bibtex \jobname}
\usepackage{color}
\definecolor{deepblue}{rgb}{0,0,0.5}
\definecolor{deepred}{rgb}{0.6,0,0}
\definecolor{deepgreen}{rgb}{0,0.5,0}

	\definecolor{DarkBlue}{rgb}{0.00,0.00,0.55}
	\definecolor{Black}{rgb}{0.00,0.00,0.00}
	\hypersetup{
		linkcolor = DarkBlue,
		anchorcolor = DarkBlue,
		citecolor = DarkBlue,
		filecolor = DarkBlue,
		colorlinks  = true,
		urlcolor = Black
	}
	
\graphicspath{ {./} }
\DeclareGraphicsExtensions{.eps,.pdf,.png,.jpg}  
	
\newtheorem{theorem}{Theorem}[section]
\newtheorem{lemma}[theorem]{Lemma}

\theoremstyle{definition}
\newtheorem{definition}{Definition}[section]

\theoremstyle{remark}
\newtheorem{remark}{Remark}[section]

\newcommand{\TheTitle}{Multilevel quasi Monte Carlo methods for elliptic PDEs with random field coefficients via fast white noise sampling}
 
\newcommand{\TheAuthors}{M.~Croci, M.~B.~Giles and P.~E.~Farrell}

\title{{\TheTitle}\thanks{\textbf{Funding:} This research is supported by EPSRC grants EP/R029423/1, and by the EPSRC Centre For Doctoral Training in Industrially Focused Mathematical Modelling (EP/L015803/1).}}

\author{
  M. Croci${}^\ddagger$\thanks{Mathematical Institute, University of Oxford, Oxford, UK. (\textbf{\url{matteo.croci@maths.ox.ac.uk}}), (\textbf{\url{patrick.farrell@maths.ox.ac.uk}}), (\textbf{\url{mike.giles@maths.ox.ac.uk}}).}
  \and
  M.~B.~Giles\footnotemark[2]
  \and
  P.~E.~Farrell\footnotemark[2]
}

\usepackage{amsopn}
\DeclareMathOperator{\R}{\mathbb{R}}
\DeclareMathOperator{\N}{\mathbb{N}}

\DeclareMathOperator{\E}{\mathbb{E}}
\DeclareMathOperator{\V}{\mathbb{V}}

\DeclareMathOperator{\spn}{\text{span}}
\DeclareMathOperator{\W}{\dot{W}}

\DeclareMathOperator{\ind}{\mathds{1}}
\DeclareMathOperator{\supp}{\text{supp}}
\DeclareMathOperator{\T}{\mathcal{T}}

\DeclareMathOperator{\lell}{t}

\usepackage{todonotes}
\definecolor{myblue}{RGB}{135, 206, 250}

\ifpdf
\hypersetup{
  pdftitle={\TheTitle},
  pdfauthor={\TheAuthors}
}
\fi

\begin{document}

\maketitle

\begin{abstract}
	When solving partial differential equations with random fields as coefficients the efficient sampling of random field realisations can be challenging. In this paper we focus on the fast sampling of Gaussian fields using quasi-random points in a finite element and multilevel quasi Monte Carlo (MLQMC) setting. Our method uses the SPDE approach of Lindgren et al.~combined with a new fast algorithm for white noise sampling which is taylored to (ML)QMC. We express white noise as a wavelet series expansion that we divide in two parts. The first part is sampled using quasi-random points and contains a finite number of terms in order of decaying importance to ensure good QMC convergence. The second part is a correction term which is sampled using standard pseudo-random numbers. We show how the sampling of both terms can be performed in linear time and memory complexity in the number of mesh cells via a supermesh construction, yielding an overall linear cost. Furthermore, our technique can be used to enforce the MLQMC coupling even in the case of non-nested mesh hierarchies. We demonstrate the efficacy of our method with numerical experiments.
\end{abstract}


\section{Introduction}

In this paper we extend the work of \cite{Croci2018} to the quasi and multilevel quasi Monte Carlo case (QMC and MLQMC respectively). We consider the solution of random elliptic partial differential equations (PDEs) in which Mat\'ern fields, sampled via the stochastic PDE (SPDE) approach \cite{Lindgren2011}, appear as coefficients. For instance, a typical problem is: find $\E[P]$, where $P(\omega)=\mathcal{P}(p)$ and $\mathcal{P}$ is a Fr\'echet differentiable functional of the function $p$ that satisfies,
\begin{align}
\label{eq:diffusion_eqn_for_QMC_conv_general}
\begin{array}{rlc}
-\nabla\cdot(F(u(\bm{x},\omega))\nabla p(\bm{x},\omega)) = f(\bm{x}), & \bm{x}\in G\subset\R^d,& \omega\in\Omega,
\end{array}
\end{align}
with suitable boundary conditions and smoothness assumptions on the function $f$ and the domain $G$. $F(x)>0$ is typically chosen in the literature to be a continuous, locally Lipschitz function (e.g.~an exponential) \cite{GrahamKuoNichols2013,KuoScheichlSchwabEtAl2017,Croci2018}. In this work, we assume that the coefficient $u(\bm{x},\omega)$ is a zero-mean Mat\'ern field approximately sampled by solving the (domain-truncated) Whittle SPDE \cite{Whittle1954, Lindgren2011}:
\begin{align}
\label{eq:white_noise_eqn}
\left(\mathcal{I} - \kappa^{-2}\Delta\right)^{k}u(\bm{x},\omega) = \eta \W,\quad \bm{x}\in D\subset\R^d,\quad \omega\in\Omega,\quad \nu = 2k - d/2 > 0,\quad \kappa = \frac{\sqrt{2\nu}}{\lambda},
\end{align}
where $\W$ is spatial Gaussian white noise in $\R^d$, defined over a suitable sample space $\Omega$. Here $k>d/4$, $d\leq 3$ and the equality has to hold almost surely (a.s.) and be interpreted in the sense of distributions. The constant $\eta>0$ is a scaling factor that depends on $\sigma$, $\lambda$ and $\nu$, cf.~\cite{Croci2018}.
In what follows we assume that $G\subset\joinrel\subset D \subset\joinrel\subset\R^d$, where by $G\subset\joinrel\subset D$ we indicate that the closure of $G$ is a compact subset of $D$, and we prescribe homogeneous Dirichlet boundary conditions on $\partial D$. If the distance between $\partial D$ and $\partial G$ is large enough, then the error introduced by truncating $\R^d$ to $D$ is negligible \cite{potsepaev2010application,Khristenko2018}.

A wide range of Gaussian field sampling methods are available in the literature, the most common being the covariance matrix factorization, the Karhunen--Lo\`eve expansion of the random field (see section 11.1 in \cite{sullivan2015introduction}), the hierarchical matrix approximation of the covariance matrix \cite{FeischlKuoSloan2018,Khoromskij2009,Hackbush2015HMatrices,DolzHarbrechtSchwab2017}, the circulant embedding method \cite{WoodChan1994,dietrich1997fast,GrahamKuoEtAl2018,BachmayrGraham2019} and the SPDE approach \cite{Whittle1954,Lindgren2011,Croci2018}. We refer to section 2.4 in \cite{CrociPHD} for a detailed overview and comparison.

In this paper, we only consider the SPDE approach, which consists of sampling a Mat\'ern field by solving equation \eqref{eq:white_noise_eqn} for a given realization of white noise. Equation \eqref{eq:white_noise_eqn} can be solved after discretization in $O(m)$ or $O(m\log(m))$ cost complexity with an optimal solver such as multigrid. This yields up to an overall $O(m)$ sampling complexity provided that the white noise term can also be sampled in linear cost. In \cite{Croci2018}, the authors have showed how white noise realizations can be sampled in $O(m)$ complexity within a finite element (FE) and non-nested MLMC framework.

The new (ML)QMC method we present in this paper is based on the efficient sampling of the white noise term in \eqref{eq:white_noise_eqn} with a hybrid quasi/pseudo-random sequence. This technique does not require structured grids and is especially designed to work with non-nested mesh hierarchies. While some sampling techniques (e.g.~Karhunen--Lo\`eve expansion or covariance matrix factorization) are easily extended to the non-nested case, the SPDE approach requires more careful handling \cite{Croci2018,Osborn2017}. To the best of our knowledge this work is the first in which the SPDE approach is employed in a non-nested MLQMC framework.

For QMC applications it is extremely important for the QMC integrand to have low effective dimensionality and to order the QMC integrand variables in order of decaying importance \cite{CaflishMorokoffOwen}. For this reason, a common approach in the existing literature about MLQMC methods for elliptic PDEs is the expansion of the random field coefficients as an infinite series of basis functions of $L^2(D)$ that naturally exposes the leading order dimensions in the integrands \cite{KuoScheichlSchwabEtAl2017,KuoSchwabSloan2015,GrahamKuoNichols2013,DickKuoSloan2013}.

When using the SPDE approach and equation \eqref{eq:white_noise_eqn}, the only source of randomness is white noise and we therefore must expand $\W$ to achieve the required variable ordering. In this paper, we use a wavelet expansion of $\W$. Wavelets in general form a multi-resolution orthogonal basis of $L^2(D)$ and are commonly employed within QMC algorithms as their hierarchical structure exposes the leading order dimensions in the integrands while allowing fast $O(m)$ or $O(m\log m)$ complexity operations (depending on the wavelet basis \cite{Daubechies1988}). A classical example on the efficacy of wavelet expansions of white noise (in time) within a QMC method is offered by the L\'evy--Ciesielski (or Brownian bridge) construction of Brownian motion. Ubiquitous in mathematical finance, it is commonly used to solve stochastic differential equations with QMC \cite{glasserman2013,GilesWaterhouse2009}. Inspired by this technique, we choose to expand white noise into a Haar wavelet expansion\footnote{Note that the hat functions used in the L\'evy--Ciesielski construction are piecewise linear wavelets, their derivatives are Haar wavelets and white noise in time is the derivative of Brownian motion.}.

In a MLQMC framework, wavelets are used by Kuo et al.~to sample random fields efficiently, yielding a cost per sample of $O(m \log m)$ using nested grids \cite{KuoSchwabSloan2015}. In \cite{HerrmannSchwab2017}, Herrmann and Schwab use a truncated wavelet expansion of white noise to sample Gaussian fields with the SPDE approach within a nested MLQMC hierarchy. Their work is possibly the closest to ours as they also work with the SPDE approach to Mat\'ern field sampling and use a wavelet expansion of white noise \cite{HerrmannSchwab2017}. All these randomized MLQMC methods for elliptic PDEs use randomly shifted lattice rules and derive MLQMC complexity bounds using a pure QMC approach, truncated expansions and nested hierarchies.


Our work differs from the above-mentioned papers as it specifically focuses on computational aspects. Firstly, we do not derive any MLQMC complexity estimates, but we design our method to work in the general case in which the multilevel hierarchy is non-nested. Secondly, our objective is to sample white noise exactly in linear (or log-linear) cost complexity within a (ML)QMC framework, similarly as for MLMC in \cite{Croci2018}. For this reason, we handle the expansion differently: we do not just truncate it, but we work with the whole infinite expansion of white noise by adding a correction term to the truncation. The truncation term is finite-dimensional and we sample it with a randomized QMC point sequence; the correction term is infinite-dimensional and a QMC approach is not feasible. However, the covariance of the correction is known and we can sample it using pseudo-random numbers with an extension of the technique presented in \cite{Croci2018}.

The advantage of our hybrid QMC/MC approach is that we can sample white noise exactly, independently from the truncation level and the wavelet degree considered (e.g.~while we use Haar wavelets, Herrmann and Schwab in \cite{HerrmannSchwab2017} consider higher degree wavelets), without introducing any additional bias into the MLQMC estimate. In contrast, in the aforementioned MLQMC algorithms the expansion must be truncated after enough terms to make the truncation error negligible. Naturally, this advantage comes at a price: since we are using pseudo-random numbers as well, the asymptotic convergence rate of our method with respect to the number of samples $N$ is still the standard MC rate of $O(N^{-1/2})$. Nevertheless, we show that large computational gains can be recovered in practice in a pre-asymptotic QMC-like regime in which the convergence rate is $O(N^{-\chi})$, $\chi \geq 1/2$, and we derive a partial convergence result (cf.~supplementary material) that explains this behaviour in the QMC case.

Wavelet spaces are typically non-nested within FE spaces and vice-versa, and a transfer operation (i.e.~interpolation, projection or quadrature) is necessarily needed somewhere along the Gaussian field sampling pipeline. While this issue is not mentioned in related work on MLQMC \cite{HerrmannSchwab2017}, transfer operations are often performed at the end of the sampling process in non-nested MLMC (see e.g.~\cite{GrahamKuoNuyensCirculantEmbeddingUnstructured2018,Osborn2017scalable}): the field is first sampled on a usually uniform and structured sampling mesh before transferring it onto the target FE grids. While sampling on a uniform structured grid has its appeal, this approach also has two disadvantages. First, when working with complex geometries and graded meshes, it is desirable for the sampled Mat\'ern field to have the same local accuracy as the solution of the PDE of interest (e.g.~\eqref{eq:diffusion_eqn_for_QMC_conv_general}). Uniform grids must be refined everywhere to satisfy this requirement. Second, the non-nested random field transfer could introduce additional bias, which may be asymptotically larger in some cases. For instance, a piecewise polynomial defined on a uniform grid is not globally differentiable and cannot be accurately integrated or interpolated on a non-nested grid without resorting to more sophisticated techniques.

Our approach solves both issues by directly sampling coupled white noise realizations exactly on each pair of non-nested FE subspaces in the hierarchy so that all the subsequent computations (solution of the SPDE and the PDE of interest, evaluation of the output functional) can be performed non-intrusively and in parallel without any further inter-level communication. This is done similarly as in our previous work \cite{Croci2018}, but with the additional feature that the white noise samples are QMC samples obtained via a wavelet expansion. For this sampling operation we adopt the embedded mesh technique by Osborn et al.~\cite{Osborn2017} so that in the MLQMC hierarchy each mesh of $G$ is nested within the corresponding mesh of $D$ and we deal with the non-nestedness of the FE and wavelet spaces via a supermesh construction. The latter is key in our algorithm for drawing exact white noise samples (over the FE spaces). In the independent white noise realization case we construct a two-way supermesh between the graded FE mesh of interest and a uniform ``wavelet'' mesh. In the MLQMC coupled realization case, we construct a three-way supermesh between the two non-nested FE meshes and the ``wavelet'' mesh. If all meshes involved are nested, a supermesh construction is not required. Owing to the embedded mesh technique and the drawing of exact white noise samples we avoid all non-nested transfer operations and any additional bias.

It is worth noting that the QMC and MLQMC methods presented are especially designed to work with non-nested grid hierarchies and that the use of a supermesh construction is merely a requirement of this framework rather than a specific feature of our algorithms. In fact, a supermesh construction might be unavoidable even when the field is sampled via an alternative method such as e.g.~circulant embedding \cite{dietrich1997fast}, since this is needed to integrate the sampled field exactly on the target non-nested unstructured grid (see \cite{CharrierMLMC2013} for an analyis of quadrature error). For low-smoothness fields ($\nu\leq 1$), Graham et al.~have proved that there is no loss in the convergence rate due to non-nested interpolation \cite{GrahamKuoNuyensCirculantEmbeddingUnstructured2018}. However, this operation, albeit faster, still introduces extra bias and harms convergence when smoother fields are used. 

We remark that the QMC method presented in this paper has been successfully applied in \cite{CrociVinjeRognes2019bis} on an unstructured MRI-derived 3D brain mesh to quantify the uncertainty in models describing brain fluid and solute movement.

This paper is structured as follows: in section \ref{sec:background} we summarize the mathematical background needed to understand the rest of the paper. In section \ref{sec:Haar_wavelet_expansion_of_white_noise} we introduce the Haar wavelet expansion of white noise and its splitting into a truncated term and a correction term. In section \ref{sec:QMC_sample_white_noise} we introduce our sampling technique for independent white noise realizations. We extend the white noise sampling method to MLQMC in section \ref{sec:coupled_realizations_MLQMC}, where we show how coupled white noise realizations can be sampled efficiently. The algorithms are supported by numerical results, which we present and discuss in section \ref{sec:MLQMC_num_res}. We conclude the paper with a brief summary of the methods and results presented in section \ref{sec:MLQMC_conclusions}.

\section{Notation and background}
\label{sec:background}

\subsection{Notation}

In this paper we denote with $L^2(D)$ the space of square-integrable functions over $D$ and with $(\cdot, \cdot)$ the standard $L^2(D)$ inner product. We furthermore indicate with $W^{k,q}(D)$ the standard Sobolev space of integrability order $q$ and differentiability $k$, with $H^k(D) \equiv W^{k,2}(D)$ and with $H^1_0(D)$ the space of $H^1(D)$ functions that vanish on $\partial D$ in the sense of traces.

Given a sample space $\Omega$, we indicate with $L^2(\Omega, \R)$ the space of real-valued \emph{random variables} with finite second moment.

For a given Banach space $U$ over $D$, we indicate with $L^2(\Omega, U)$ the space of \emph{random fields} $u(\bm{x}, \omega)$, $\bm{x}\in D$, $\omega\in\Omega$ such that $u(\bm{x}, \cdot)\in L^2(\Omega, \R)$ for almost every $\bm{x}\in D$ and $u(\cdot, \omega)\in U$ almost surely. If the $u(\bm{x},\cdot)$ are jointly Gaussian for almost every $\bm{x}\in D$, then the random field is a \emph{Gaussian field} and it is uniquely determined by its mean $\mu(\bm{x})$ and covariance $C(\bm{x},\bm{y})$ function. Throughout this paper we will consider only zero-mean fields for simplicity.

A Gaussian field is also a \emph{Mat\'ern field} if its covariance is of the Mat\'ern class, i.e.
\begin{align}
\label{eq:Matern}
C(\bm{x},\bm{y}) := \dfrac{\sigma^2}{2^{\nu-1}\Gamma(\nu)}(\kappa r)^\nu \mathcal{K}_\nu(\kappa r),\ \ r:=\Vert \bm{x}-\bm{y}\Vert _2,\ \ \kappa := \frac{\sqrt{2\nu}}{\lambda},\ \ \bm{x},\bm{y}\in D,
\end{align}
where $\sigma^2$, $\nu$, $\lambda > 0$ are the variance, smoothness parameter and correlation length of the field respectively, $\Gamma(x)$ is the Euler Gamma function and $\mathcal{K}_\nu$ is the modified Bessel function of the second kind.

In this paper we will adopt the following definition of \emph{generalized random field}, first introduced by It\^{o} \cite{Ito1954} and extended by Inaba and Tapley \cite{Inaba1975}. For a given Banach space $U$, we denote with ${\mathcal{L}(U, L^2(\Omega, \R))}$ the space of generalized random fields that are continuous linear mappings from $U$ to $L^2(\Omega,\R)$. For a given $\xi\in \mathcal{L}(U, L^2(\Omega, \R))$ we indicate the action (or pairing) of $\xi$ onto a function $\phi\in U$ with the notation $\xi(\phi) := \langle \xi, \phi \rangle$.

Possibly the most commonly used generalized random field is Gaussian \textit{white noise}. White noise is defined as follows.
\begin{definition}[White noise, see example 1.2 and lemma 1.10 in \cite{Hida1993}]
	\label{def:white_noise}
	Let $D\subseteq\R^d$ be an open domain. The white noise $\W\in \mathcal{L}(L^2(D), L^2(\Omega, \R))$ is a generalized stochastic field such that for any collection of $L^2(D)$ functions $\{\phi_i\}$, if we let $b_i := \langle \W, \phi_i \rangle$, then $\{b_i\}$ are joint Gaussian random variables with zero mean and covariance given by $\E[b_ib_j]=(\phi_i,\phi_j)$.
\end{definition}

When working with the FEM, it is useful to define the $L^2$ projection of white noise onto a FE space and its properties:
\begin{definition}[$L^2$ projection of white noise, definition 3.2 in \cite{CrociPHD}]
	Let $D_h$ be a FE mesh of an open domain $D\subseteq\R^d$ and let $U_h\subseteq L^2(D)$ be a FE subspace defined over $D_h$. Let $P_h$ be the $L^2$ projection onto $U_h$. The $L^2$ projection of white noise $P_h\W\in \mathcal{L}(L^2(D), L^2(\Omega, \R))$ is the generalized random field that satisfies
	\begin{align}
	\langle P_h\W,v\rangle:=\W(P_hv)\equiv \langle \W, P_hv\rangle,\quad\forall v\in L^2(D).
	\end{align}
\end{definition}
\begin{lemma}
	If $\text{dim}(U_h)<\infty$, then $P_h\W$ is a ``proper'' Gaussian field in $L^2(\Omega,U_h)\subseteq L^2(\Omega, L^2(D))$ and $\langle P_h\W,v\rangle\equiv (P_h\W,v)$ for all $v\in L^2(D)$. Furthermore,
	\begin{align}
	(P_h\W,v_h) =  \langle \W, v_h\rangle,\quad\forall v_h\in U_h,\quad\quad (P_h\W,v^\perp) = 0,\quad \forall v^\perp\in U_h^\perp,
	\end{align}
	i.e.~$\W_h$ coincides with $\W$ over $U_h$ and is zero on $U_h^\perp$.
\end{lemma}
\begin{proof}
	Every $v\in L^2(D)$ admits the decomposition $v = v_h + v^\perp$, where $v_h\in U_h$ and $v^\perp \in U_h^\perp$. Given two orthonormal bases $\{\phi_i\}_{i=1}^\infty$ of $U_h$ and $\{\phi^\perp_i\}_{i=1}^\infty$ of $U_h^\perp$, we can then expand the action of white noise against any $v\in L^2(D)$ as the series
	\begin{align}
	\langle \W, v\rangle = \sum_{i=1}^n a_i(\phi_i,v) + \sum_{i=1}^\infty a_i^\perp (\phi_i^\perp,v),\quad\text{where}\quad a_i = \langle \W, \phi_i\rangle,\quad a_i^\perp=\langle \W, \phi_i^\perp\rangle.
	\end{align}	
	By defining $\W_h:=\sum_{i=1}^na_i\phi_i\in L^2(\Omega,U_h)$, it is then readily shown that $\W_h\equiv P_h\W$ in $L^2(\Omega,L^2(D))$ since $(\W_h,v^\perp)=0$ a.s.~for all $v^\perp\in U_h^\perp$, and therefore
	\begin{align}
	\langle P_h\W, v\rangle := \langle \W, P_hv\rangle = \sum_{i=1}^n a_i(\phi_i,P_hv)=(\W_h,v),\quad\text{a.s.},\quad\forall v\in L^2(D).
	\end{align}
	The remaining part of the thesis readily follows.
\end{proof}

\subsection{Randomized quasi Monte Carlo}

Quasi Monte Carlo (QMC) methods retain most of the advantages of standard Monte Carlo (MC) while improving the convergence order with respect to the number of samples. At the heart of QMC for estimating expectations is the reinterpretation/approximation of the expected value as an integral with respect to the uniform distribution over the unit hypercube. This integral is then approximated with an $N$-point quadrature rule with equal weights:
\begin{align}
\label{eq:QMC_intro}
\E[P] = \int\limits_\Omega P(\omega) \text{d}\mathbb{P}(\omega) \approx \int\limits_{[0,1]^s}Y(\bm{x})\text{d}\bm{x} =: I \approx \frac{1}{N}\sum\limits_{n=1}^NY(\bm{x}_n) =: I_N,
\end{align}
for some suitable function $Y$ obtained by mapping the integration domain to the unit hypercube (cf.~section 2.11 in \cite{DickKuoSloan2013}). Here the first approximation sign indicates that there might be an additional error due to dimension truncation in case the integrand is infinite-dimensional. The second approximation sign simply represent the actual QMC approximation. Here $\mathbb{P}$ is the probability measure of a suitable probability space $(\Omega, \mathcal{A}, \mathbb{P})$, with $\mathcal{A}$ as its $\sigma$-algebra. The $\bm{x}_n\in[0,1]^s$ are, unlike in the standard MC case, not chosen at random, but chosen carefully and in a deterministic way so as to minimise the quadrature error. However, these deterministic points often yield error bounds which are difficult to compute. For this reason QMC points are typically randomized to allow a more practical statistical error estimation \cite{Lemieux2009,Owen2003}. In randomized QMC, $M$ independent randomized sequences $\{\{\hat{\bm{x}}_{n,m}\}_{n=1}^N\}_{m=1}^{M}$ are used to obtain the estimator
\begin{align}
\hat{I}_{N,M}(\omega) := \frac{1}{M}\sum\limits_{m=1}^M I^m_N(\omega) = \frac{1}{M}\sum\limits_{m=1}^M \left(\frac{1}{N}\sum\limits_{n=1}^NY(\hat{\bm{x}}_{n,m}(\omega))\right).
\end{align}
Since $I^m_N(\omega)$ is now a random variable, provided that $M$ is large enough a confidence interval can be estimated and we retain a practical error measure as in the standard MC case. In this work we use $M=32$ unless otherwise stated.

QMC methods achieve faster-than-MC convergence rates by carefully choosing the point sequence $\{\bm{x}_n\}_{n=1}^{N}$ and its randomization. For a detailed overview of QMC integration theory we refer the reader to the excellent review article by Dick et al.~\cite{DickKuoSloan2013}. 

According to both classical \cite{Lemieux2009,CaflishMorokoffOwen} and modern \cite{KuoSchwabSloan2015,GrahamKuoNuyens2011,GrahamKuoNuyensCirculantEmbeddingUnstructured2018,KuoScheichlSchwabEtAl2017,HerrmannSchwab2017} QMC theory, for practical applications of high-dimensional QMC integration it is extremely important to order the integration variables in order of decaying importance and/or possibly reduce the dimensionality of the integrand. This will be a key aspect in the methods we present in this paper.

\subsection{Multilevel Monte Carlo methods}

The multilevel Monte Carlo method was first introduced by Heinrich in \cite{Heinrich2001} for parametric integration and popularised by Giles for stochastic differential equations in \cite{giles2008}. Multilevel quasi Monte Carlo, originally presented in \cite{GilesWaterhouse2009}, combines QMC and MLMC together with the objective of combining their advantages. Assume that it is possible to compute realizations of $P(\omega)$ at different accuracy levels $P_\ell(\omega)$ for $\ell=1,\dots,L$ of increasing accuracy and computational cost, and that the approximation of $P$ on the finest level, $P_L$ is accurate enough. Multilevel methods estimate $\E[P_L]$ through the telescopic sum,
\vspace{-6pt}
\begin{align}
\label{eq:telescoping}
\E[P] \approx \E[P_L] = \sum\limits_{\ell = 1}^L\E[P_\ell - P_{\ell-1}],
\end{align}
\vspace{-9pt}\\
where $P_{0} \equiv 0$. For example, if samples of $P$ are obtained by soving \eqref{eq:diffusion_eqn_for_QMC_conv_general} with the finite element method (FEM), the levels of accuracy can be defined by using a hierarchy of meshes ($h$-refinement) or by increasing the polynomial degree of the finite element bases used ($p$-refinement).

The MLMC and MLQMC estimators are then obtained from \eqref{eq:telescoping} by approximating each term in the sum with standard MC or randomized QMC respectively. For MLMC we have,
\begin{align}
\E[P_\ell - P_{\ell-1}] \approx \frac{1}{N_\ell}\sum\limits_{n=1}^{N_\ell}(P_\ell - P_{\ell  - 1})(\omega^n_\ell),
\end{align}
in which each $P_\ell - P_{\ell-1}$ sample is coupled in the sense that the samples of $P_\ell$ and $P_{\ell-1}$ share the same event $\omega_\ell^n$. Ensuring this coupling is respected is essential for any MLMC algorithm since this coupling is the reason behind the increased efficiency of MLMC with respect to standard MC \cite{giles2008}.

Enforcing the same type of coupling is also essential for MLQMC. In the MLQMC case each term in \eqref{eq:telescoping} is approximated with randomized QMC as follows,
\begin{align}
\label{eq:mlqmc_level_estimator}
\E[P_\ell-P_{\ell-1}] \approx \int\limits_{[0,1]^{s_\ell}}Y_\ell(\bm{x})\text{d}\bm{x}\approx \frac{1}{M}\sum\limits_{m=1}^M \left(\frac{1}{N_\ell}\sum\limits_{n=1}^{N_\ell}Y_\ell(\hat{\bm{x}}_{n,m}^\ell)\right)=:\frac{1}{M}\sum\limits_{m=1}^MI_{N_\ell}^{m,\ell}(\omega),
\end{align}
where the meaning of each variable is the same as in the QMC case. We now have a hierarchy of integrands $\{Y_\ell\}_{\ell=1}^L$ and of randomized QMC point sequences $\{\{\{\hat{\bm{x}}_{n,m}^\ell\}_{n=1}^N\}_{m=1}^M\}_{\ell=1}^{L}$ of dimensions $\{s_\ell\}_{\ell=1}^L$. Note that, in the same way as for QMC, MLQMC still requires for good performance that the integration variables on each level are organized in order of decaying importance and possibly that the integrands $Y_\ell$ have low effective dimensionality.

The theory for MLMC is by now established \cite{giles2008,Cliffe2011,TeckentrupMLMC2013}, yielding formulas for the optimal number of samples $N_\ell$ on each level and for the total MLMC algorithm complexity ($O(\varepsilon^{-2})$ in the best case scenario, see supplementary material \ref{secSM:multilevel_methods}). On the other hand, obtaining a reliable estimate of the QMC convergence rate with respect to $N_\ell$ for MLQMC is particularly hard, to the extent that theoretical results are only available for a few specific problems and specific QMC point sequences \cite{KuoScheichlSchwabEtAl2017,HerrmannSchwab2017}. For this reason, setting up an optimal MLQMC hierarchy with the optimal values of the $N_\ell$ is a challenging task. However, in the best possible case where we get a $O(N^{-\chi})$, $1/2 \leq \chi \leq 1$, QMC rate for each term in the telescoping sum, the benefits of MLMC and QMC can accumulate yielding a total MLQMC computational cost of $O(\varepsilon^{-1/\chi})$ for a given MSE tolerance of $\varepsilon^2$ \cite{HerrmannSchwab2017}. In this case MLQMC significantly outperforms all other Monte Carlo methods.

In this paper we employ the original MLQMC algorithm from \cite{GilesWaterhouse2009} as it does not require the convergence rate with respect to $N$ to be known \emph{a priori}. We refer to the supplementary material \ref{secSM:multilevel_methods} for a description of the algorithm.

\subsection{Supermeshes}

We now introduce the concepts of non-nested tessellations/meshes and of a supermesh.

\begin{figure}[h!]
	\centering
	\includegraphics[trim=0cm 0cm 0cm 0cm, clip=true, width=.7\textwidth]{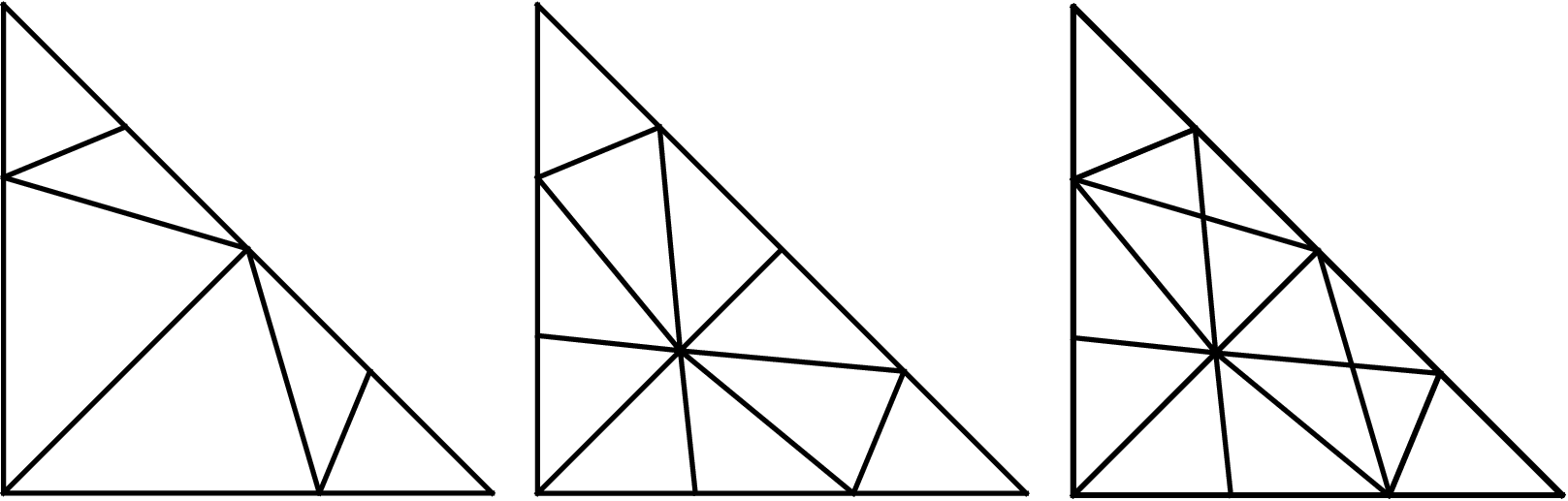}
	\centering
	\caption{\textit{An example of a supermesh construction. The first two meshes on the left are the parent meshes and the mesh on the right is one of their supermeshes.}}
	\label{fig:supermesh}
	\vspace{-12pt}
\end{figure}

Let $T_a$ and $T_b$ be two tessellations of $D$. We say that $T_a$ is \emph{nested} within $T_b$ if $\textnormal{vertices}(T_a)\subseteq \textnormal{vertices}(T_b)$ and if for each element $e\in T_a$ there exists a set of elements $E\subseteq T_b$ such that $e=\bigcup_{\hat{e}_i\in E}\hat{e}_i$. We say that $T_a$ and $T_b$ are \emph{non-nested} if $T_a$ is not nested within $T_b$ and vice-versa.

To enforce the multilevel (quasi) Monte Carlo coupling across a non-nested mesh hierarchy we use a crucial ingredient: a supermesh construction. Supermeshes are commonly used e.g.~within adaptive discretizations or geometric multigrid algorithms, to transfer discrete fields between non-nested meshes \cite{Farrell2009Supermesh}. A supermesh is defined as follows.

\begin{definition}[Supermesh, \cite{Farrell2011Supermesh, Farrell2009Supermesh}]
	Let $D\subset\joinrel\subset\R^d$ be an open domain and let $\T_a$, $\T_b$ be two tessellations of $D$. A supermesh $S_h$ of $\T_a$ and $\T_b$ is a common refinement of $\T_a$ and $\T_b$. More specifically, $S_h$ is a triangulation of $D$ such that:
	\begin{enumerate}[leftmargin=1cm]
		\item $\textnormal{vertices}(\T_a) \cup \textnormal{vertices}(\T_b) \subseteq \textnormal{vertices}(S_h)$,
		\item $\textnormal{measure}(e_S \cap e) \in \{0, \textnormal{measure}(e_S)\}$ for all cells $e_S\in S_h$, $e\in (\T_a \cup \T_b)$.
	\end{enumerate}
\end{definition}

The first condition means that every parent mesh vertex must also be a vertex of the supermesh, while the second states that every supermesh cell is completely contained within exactly one cell of either parent mesh \cite{Farrell2009Supermesh}. The supermesh construction is not unique \cite{Farrell2009Supermesh}. We show an example of supermesh construction in figure \ref{fig:supermesh}. Efficient algorithms for computing the supermesh are available \cite{libsupermesh-tech-report}.

It was shown in \cite{quasi-uniform-supermeshing} that if the parent tesselations $T_a$ and $T_b$ are quasi-uniform (cf.~definition 4.4.13 in \cite{brenner2007mathematical}), then the number of cells of a supermesh constructed via a local supermeshing algorithm (cf.~\cite{PatrickPHD}) is linear in the number of cells of $T_a$ and $T_b$.

\subsection{Fast white noise sampling over finite element spaces}
\label{subsec:fast_WN_sampling}
When working with Monte Carlo methods and the SPDE approach, it is essential to draw white noise samples efficiently. In practice, it is typically sufficient to sample realizations of $P_h\W$ since $P_h\W$ coincides with $\W$ over the FE space $U_h=\spn\{\phi_1,\dots,\phi_n\}$. In turn, samples of $P_h\W$ can be obtained by drawing realizations of the Gaussian vector $\bm{b}\in L^2(\Omega,\R^n)$ with $b_i:=(P_h\W, \phi_i)=\langle \W, \phi_i \rangle$, i.e.~the action of $P_h\W$ over the basis functions of $U_h$. The vector $\bm{b}$ has covariance matrix $M$ where $M_{ij}=(\phi_i,\phi_j)$ (cf.~definition \ref{def:white_noise}). Drawing realizations of $\bm{b}$ can be expensive since it typically involves the factorization of $M$ (the mass matrix over $U_h$) that has up to cubic cost complexity. In \cite{Croci2018}, we presented a linear (optimal) cost complexity sampling algorithm which we use in this paper and we now briefly summarize (see section 4.1 in \cite{Croci2018} for further details).

We first define the auxiliary space $\dot{U}_h$ by restricting the support of the basis functions of $U_h$ onto $\{e_k\}_{k=1}^m$, the cells of $D_h$. Namely, we define $\dot{U}_h:=\sum_{k=1}^mU_h^b|_{e_k}$ where $U_h^b|_{e_k}=\spn\{\phi_1|_{e_k},\dots,\phi_n|_{e_k}\}$. We indicate with $\{\psi_i\}_{i=1}^{\bar{n}}$ the basis functions of $\dot{U}_h$. Note that by construction $U_h\subseteq \dot{U}_h\subseteq L^2(D)$. After letting $P_h$ and $\dot{P}_h$ be the $L^2$ projections onto $U_h$ and $\dot{U}_h$ respectively, the technique in \cite{Croci2018} essentially samples $P_h\W$ as $P_h\W = P_h\dot{P}_h\W$ which is done by first sampling $\dot{P}_h\W$ and then projecting it onto $U_h$. Both operations can be performed in linear cost. Let $\bm{a}\in L^2(\Omega, \R^{\bar{n}})$, $\bm{b}\in L^2(\Omega,\R^n)$ be Gaussian vectors such that $a_i=(\dot{P}_h\W,\psi_i)$ and $b_i=(P_h\W,\phi_i)$. The sampling of $\bm{a}$ has linear cost complexity in $m$ and is trivially parallelizable since its covariance matrix is the mass matrix over $\dot{U}_h$, which is block-diagonal (it has $m$ blocks of small size) and can be factorized efficiently. We can then obtain $\bm{b}$ from $\bm{a}$ in linear cost since every basis function of $U_h$ can be written as a small sum of the basis functions of $\dot{U}_h$, and therefore there exists a sparse boolean matrix $L\in\R^{n\times\bar{n}}$ such that $\bm{b}=L\bm{a}$ almost surely. As an example, if we split a $\phi_i\in U_h$ into $\phi_i=\phi_i|_{e_j}+\phi_i|_{e_k}$, then $\phi_i|_{e_j},\phi_i|_{e_k}\in\dot{U}_h$ and
\begin{align}
(P_h\W,\phi_i)=\langle\W,\phi_i\rangle = \langle\W,\phi_i|_{e_j}\rangle + \langle\W,\phi_i|_{e_k}\rangle = (\dot{P}_h\W,\phi_i|_{e_j}) + (\dot{P}_h\W,\phi_i|_{e_k}).
\end{align}
The matrix $L$ is called an assembly matrix and its action encodes the assembly procedure performed by any FE code to build global matrices and vectors from local cell contributions \cite{Wathen1987}. Once a sample of $\bm{b}$ has been drawn, the coefficients in the FE expansion of $P_h\W$ are given by $M^{-1}\bm{b}$. Inverting the mass matrix is still a linear cost operation, although it is not needed in the SPDE approach which only requires samples of $\bm{b}$.

In a MLMC framework, we must couple samples of white noise over two different, possibly non-nested FE spaces $V^\ell$ and $V^{\ell-1}$. The technique just described can be extended so that coupled white noise samples can also be drawn in linear cost. The idea is to set $U_h=V^\ell + V^{\ell-1}$ and use this technique to sample white noise over $U_h$. Since $V^\ell,V^{\ell-1}\subseteq U_h$ one can then easily transfer the result over $V^\ell$ and $V^{\ell-1}$, and obtain coupled samples. The main complication here is that to construct the auxiliary space $\dot{U}_h$ now a supermesh construction is required: by splitting the support of all basis functions of $U_h$ across the supermesh cells we can ensure that the mass matrix over the space $\dot{U}_h$ still has a block-diagonal structure and can thus be factorized efficiently. We refer to section 4.2 in \cite{Croci2018} for further details.

\section{Haar wavelet expansion of spatial white noise}
\label{sec:Haar_wavelet_expansion_of_white_noise}

For good QMC convergence we need to order the dimensions of the QMC integrand in order of decaying importance so that the largest error components are on the leading dimensions \cite{glasserman2013,DickKuoSloan2013}. In what follows we expand white noise into a Haar wavelet series so that the hierarchical structure of Haar wavelets can naturally provide the variable ordering needed for QMC integration.

We start by briefly introducing the Haar wavelet basis, see e.g.~\cite{Daubechies1988}. Let $\ind_A(x)$ be the indicator function of a set $A$, let $\bar{\N} := \{-1\} \cup \N$ and let $x^+ := \max(x,0)$. The Haar wavelets $H_{l,n}$ for $l\in\bar{\N}$, $n=0,\dots,2^{l^+}-1$ are obtained from the mother wavelet $\Psi(x) := \ind_{[0,1/2)}(x) - \ind_{[1/2,1)}(x)$ through shifting and rescaling as $H_{l,n}:=2^{l^+/2}\Psi(2^l x-n)$. These wavelets have support of size $|\supp(H_{l,n})| = 2^{-l^+}$ and form an orthonormal basis of $L^2(0,1)$. The Haar system can be generalized to higher dimensions by taking the tensor product of the 1D Haar basis with itself. In the higher dimensional case we then have vector-indexed $d$-dimensional Haar wavelets $H_{\bm{l},\bm{n}}(\bm{x}) := \bigotimes_{i=1}^d H_{l_i, n_i}(x_i)$ ,where $\bm{l} \in \bar{\N}^d$ and $\bm{n}\in\N$. The $d$-dimensional Haar wavelets have support size $|\supp(H_{\bm{l},\bm{n}})|=\prod_{i=1}^{d} 2^{-l_i^+} = 2^{-||\bm{l}^+||_1}$ and they form an orthonormal basis of $L^2((0,1)^d)$. Generalizing the Haar basis to generic boxed domains is straight-forward.

\begin{remark}[On the non-standard Haar wavelet basis]
	\label{rem:compactly_supported_Haar_basis}
	The $d$-dimensional wavelets just introduced are sometimes called the \emph{standard} Haar basis, which leads to log-linear complexity operations (rather than just linear) for $d>1$. The algorithms we introduce in this paper also work for the \emph{non-standard} Haar basis, which supports linear complexity operations in all dimensions \cite{Daubechies1988,Beylkin1992}.
\end{remark}

Let $z_{\bm{l},\bm{n}}(\omega)$ be i.i.d.~standard normal random variables. Furthermore, let $|\bm{l}| := \max_i(l_i)$. We can express white noise over $[0,1]^d$ as a Haar wavelet expansion,
\begin{align}
\label{eq:Haar_expansion_full}
\W = \sum\limits_{|\bm{l}|=-1}^{\infty}\sum\limits_{\bm{n}=\bm{0}}^{2^{\bm{l}^+}-\bm{1}}z_{\bm{l},\bm{n}}(\omega)H_{\bm{l},\bm{n}}(\bm{x}),
\end{align}
where $z_{\bm{l},\bm{n}}:=\langle\W,H_{\bm{l},\bm{n}}\rangle$ for all $\bm{l}$, $\bm{n}$, making the $z_{\bm{l},\bm{n}}$ be i.i.d.~standard Gaussian random variables. The second summation is to be interpreted as the sum over all $\bm{n}$ with components $n_i$ such that $0\leq n_i \leq 2^{l^+}_i - 1$ for all $i$. Let $\mathscr{L}\in\bar{\N}$. We now divide the series in two terms,
\begin{align}
\label{eq:wavelet_QMC_white_noise}
\W = \W_\mathscr{L} + \W_R = \sum\limits_{|\bm{l}|=-1}^{|\bm{l}|=\mathscr{L}}\sum\limits_{\bm{n}=\bm{0}}^{2^{\bm{l}^+}-\bm{1}}z_{\bm{l},\bm{n}}(\omega)H_{\bm{l},\bm{n}}(\bm{x}) + \sum\limits_{|\bm{l}|=\mathscr{L}+1}^{\infty}\sum\limits_{\bm{n}=\bm{0}}^{2^{\bm{l}^+}-\bm{1}}z_{\bm{l},\bm{n}}(\omega)H_{\bm{l},\bm{n}}(\bm{x}).
\end{align}
Here the term $\W_\mathscr{L}$ is a ``proper'' stochastic field in $L^2(\Omega, L^2(D))$, while $\W$ and $\W_R$ are generalised stochastic fields with actions defined for all $v\in L^2(D)$ respectively by
\begin{align}
\label{eq:W_R_typeI}
\langle\W,v\rangle = \sum\limits_{|\bm{l}|=-1}^{\infty}\sum\limits_{\bm{n}=\bm{0}}^{(2^{\bm{l}}-\bm{1})^+}z_{\bm{l},\bm{n}}(H_{\bm{l},\bm{n}},v),\quad\text{and}\quad\langle\W_R,v \rangle = \sum\limits_{|\bm{l}|=\mathscr{L}+1}^{\infty}\sum\limits_{\bm{n}=\bm{0}}^{(2^{\bm{l}}-\bm{1})^+}z_{\bm{l},\bm{n}}(H_{\bm{l},\bm{n}},v).
\end{align}

The idea is then to sample the Gaussian variables in the expression for $\W_\mathscr{L}$ by using a hybrid QMC/MC combination of quasi-random (e.g.~Sobol) and pseudo-random numbers, and to sample $\W_R$ with pseudo-random numbers only by extending the work in \cite{Croci2018}.

The reasoning behind this splitting is that it is important to keep the dimensionality of the QMC point sequence relatively low: first, as we will see in the next section, the sampling of $\W$ expressed this way requires a supermesh construction and smaller dimensions imply faster $\W$ samples; second, some QMC point sequences cannot readily or efficiently be sampled in high dimensions\footnote{For example, the state-of-the-art Sobol' sequence generator, Broda, can generate the largest dimensional Sobol' sequences with $65536$ dimensions \cite{Sobol2011}. This might still be too low for an infinite-dimensional PDE setting.} and third, the approximation properties of some quasi-random sequences deteriorate as the dimensionality grows \cite{glasserman2013,DickKuoSloan2013}.

\section{Sampling independent realizations for QMC}
\label{sec:QMC_sample_white_noise}
To sample $u(\bm{x}, \omega)$, we must solve equation \eqref{eq:white_noise_eqn}. In what follows, we set $\eta=1$ and we will only consider the $k=1$ case for simplicity. We refer to \cite{Bolin2017}, \cite{Bolin2017SPDEApproach} for the general $k>d/4$ case. After these simplifications, we obtain
\begin{align}
\label{eq:white_noise_SPDE_reminder}
u - \kappa^{-2}\Delta u = \W,\quad \bm{x}\in D,\quad\text{with}\quad
u = 0,\quad \bm{x}\in\partial D.
\end{align}

From now on we introduce the simplifying assumption that $D=[0,1]^d$. Relaxing this assumption to general boxed domains is straightforward, but considering more general cases is non-trivial. We remark that the extended domain $D$ is just a tool needed for the SPDE approach and can be chosen arbitrarily. We are therefore free to choose any domain shape for $D$ anyways \cite{Lindgren2011}. Hence working with boxed domains is not really a restriction.
It is useful for what comes next to introduce the concept of a Haar mesh and a Haar space:
\begin{definition}[Haar mesh]
	\label{def:Haar_mesh}
	Let $D=[0,1]^d$ and let $\mathscr{L}\in\bar{\N}$. The Haar mesh $D_\mathscr{L}$ is the uniform quadrilateral mesh of $D$ whose cells are all regular polyhedra of volume $|\square_H| = 2^{-d(\mathscr{L}+1)}$. Note that for a given $\mathscr{L}$ there are exactly as many cells in $D_\mathscr{L}$ as terms in the wavelet expansion \eqref{eq:wavelet_QMC_white_noise} for $\W_\mathscr{L}$, namely $\mathscr{N}_\mathscr{L}=2^{d(\mathscr{L}+1)}$.
\end{definition}
\begin{definition}[Haar space]
	\label{def:Haar_space}
	Given a Haar mesh $D_\mathscr{L}$ with cells $\{\square_k\}_{k=1}^{\mathscr{N}_\mathscr{L}}$, we define its Haar space $V_H$ to be the space of piecewise constant functions over it; $V_H:=\spn(\psi_1,\dots,\psi_{\mathscr{N}_\mathscr{L}})$ with $\psi_k:=\ind_{\square_k}$. 
\end{definition}
We solve \eqref{eq:white_noise_SPDE_reminder} with the FEM. Let $D_h$ be a mesh of $D$, not necessarily nested within the Haar mesh $D_\mathscr{L}$. Let $V\subseteq H^1_0(D)$ and let $V_h=\spn(\varphi_1,\dots,\varphi_{\bar{m}})\subseteq V$ be the FE subspace used to solve equation \eqref{eq:white_noise_SPDE_reminder} on $D_h$. 
In what follows we will refer to $D_h$ as the \emph{FE mesh} and we assume for simplicity that there are always at most $m_e$ degrees of freedom of $V_h$ on each cell of $D_h$.

A discrete weak form of \eqref{eq:white_noise_SPDE_reminder} then reads: find $u_h\in V_h$ such that
\begin{align}
(u_h, v_h) + \kappa^{-2}(\nabla u_h,\nabla v_h) = \langle \W, v_h \rangle
\quad \text{for all } v_h\in V_h.
\end{align}
The solution ${u_h=\sum_{i=1}^{\bar{m}}u_i\varphi_i}$, expressed in terms of the basis functions of $V_h$, is given by the following linear system for the $u_i$,
\begin{align}
\label{eq:linear_system}
A\bm{u} = \bm{b},\quad\text{where}\quad A_{ij}:=(\varphi_i,\varphi_j) + \kappa^{-2}(\nabla\varphi_i,\nabla\varphi_j),\quad b_i := \langle \W, \varphi_i \rangle.
\end{align}
Now, since $\W = \W_{\mathscr{L}_\ell} + \W_R$, the $b_i$ can also be expressed as
\begin{align}
\label{eq:bL_bR}
b_i = (\bm{b}_\mathscr{L})_i + (\bm{b}_R)_i,\quad\text{with}\quad (\bm{b}_\mathscr{L})_i = (\W_\mathscr{L}, \varphi_i),\quad (\bm{b}_R)_i = \langle\W_R, \varphi_i\rangle,
\end{align}
Note that we use the $L^2(D)$ inner product notation for $\W_{\mathscr{L}}$ since $\W_{\mathscr{L}}$ is in $L^2(\Omega,L^2(D))$.

The task of computing a realization of white noise is therefore equivalent to computing a sample of $\W_{\mathscr{L}}$ and $\W_R$, and consequently of the two vectors $\bm{b}_{\mathscr{L}}$ and $\bm{b}_R$. As we will see, the sampling strategies for the two terms are considerably different. Nevertheless, we will explain how both terms can be sampled efficiently in linear or log-linear complexity.

Whenever the FE mesh is non-nested within the Haar mesh, our sampling algorithm requires the construction of a supermesh between the FE and Haar meshes so as to split the support of the wavelets and FE basis functions into smooth portions that can be integrated with high accuracy. Formally, this means that in the non-nested case we construct an auxiliary space $\dot{V}_h\subseteq L^2(D)$ as follows: first, we augment $V_h$ with FE basis functions of the same type with nodal values on $\partial D$ (we take these to be $b\in\N$) to obtain the intermediate space $V_h^b=\spn(\varphi_1,\dots,\varphi_{\bar{m}},\varphi_{{\bar{m}}+1},\dots,\varphi_{{\bar{m}}+b})\subseteq H^1(D)$. Then, we define $\dot{V}_h$ by restricting the support of the basis functions of $V_h^b$ to the Haar cells so that $\dot{V}_h=\sum_{k=1}^{\mathscr{N}_\mathscr{L}}V_h^b|_{\square_k}$, where $V_h^b|_{\square_k}=\spn(\varphi_1|_{\square_k},\dots,\varphi_{{\bar{m}}+b}|_{\square_k})$. Note that each basis function of $\dot{V}_h$ is entirely contained within a single Haar cell by construction. In this work we assume that constants can be exactly represented by the FE space\footnote{This is a standard assumption that is satisfied by most FE families, see e.g.~\cite{brenner2007mathematical}.} $V_h^b$ (i.e. $c\in V_h^b$ for all $c\in \R$) so that we have $V_h,V_H\subseteq \dot{V}_h$. 


\begin{remark}
	In what follows we will denote with $S_h$ a given supermesh between $D_h$ and $D_\mathscr{L}$. Furthermore, we will indicate and with $\{\phi_i\}_{i=1}^m$ ($m\in\N$) the basis functions of $\dot{V}_h$.
\end{remark}

We can now proceed with the description of our sampling algorithms.

\subsection{Sampling of $\W_\mathscr{L}$}
We first consider the efficient sampling of $\W_\mathscr{L}$.
Let $V_H$ be the Haar space over the Haar mesh $D_\mathscr{L}$. It turns out that $\W_\mathscr{L}\in V_H$ almost surely and that therefore it can be expressed in terms of the basis functions of $V_H$ as $\W_\mathscr{L} = \sum_{k=1}^{\mathscr{N}_\mathscr{L}}w_k\psi_i$, where $w_k$ is the value of $\W_\mathscr{L}$ over the Haar cell $\square_k$. In practice, rather than computing the inner products of each Haar wavelet with the basis functions of $\dot{V}_h$, it is more straightforward to just compute each entry of $\bm{b}_\mathscr{L}$ as $(\bm{b}_\mathscr{L})_i = w_{\kappa(i)}\int_D\phi_i \text{ d}\bm{x}$. Here $\kappa(i)$ is the index $k$ of the Haar cell that contains the support of $\phi_i$ and $w_{\kappa(i)}$ must be computed from each sample of the coefficients in the expansion for $\W_\mathscr{L}$. Before explaining how this is actually done in practice, we prove that $\W_\mathscr{L}$ is the $L^2$ projection of white noise onto $V_H$ and therefore $\W_\mathscr{L}$ does indeed belong to $V_H$.

\begin{lemma}
	\label{lemma:W_L_is_projection}
	Let $V_H:=\spn(\psi_1,\dots,\psi_{\mathscr{N}_\mathscr{L}})$ with $\psi_i:=\ind_{\square_i}$ be the space of piecewise constant functions over the cells $\square_i$ of the Haar mesh $D_\mathscr{L}$. Let $P_H$ be the $L^2$ projection onto $V_H$. Then we have that $\W_\mathscr{L} \equiv P_H\W$ in $L^2(\Omega, L^2(D))$.
\end{lemma}
\begin{proof}
	We note that all the Haar wavelets in the expansion for $\W_\mathscr{L}$ can be represented as a linear combination of basis functions of $V_H$. Since there are exactly as many wavelets as basis functions of $V_H$ (see definition \ref{def:Haar_mesh}) and since these wavelets are linearly independent, we conclude that the Haar wavelets form a basis of $V_H$. Therefore $\W_\mathscr{L}\in L^2(\Omega,V_H)$ and $\langle \W_R, v\rangle = 0$ for all $v \in V_H$ (cf.~equation \eqref{eq:W_R_typeI}). Furthermore, for all $v\in L^2(D)$,
	\begin{align}
	(\W_\mathscr{L}, v) = (\W_\mathscr{L}, P_Hv + v^\perp) = (\W_\mathscr{L}, P_Hv) = \langle \W, P_Hv \rangle =: (P_H\W, v),
	\end{align}
	almost surely since $\langle \W_R, P_Hv\rangle = 0$ for all $v\in L^2(D)$ (cf.~equation \eqref{eq:W_R_typeI}). Here we used the fact that all $v\in L^2(D)$ can be split as $v = P_Hv + v^\perp$, where $v^\perp \in V_H^\perp$.
\end{proof}

Note that the dimension of the space $V_H$ is $\mathscr{N}_\mathscr{L}=2^{d(\mathscr{L}+1)}$, which can easily become a large number. If we were to sample $\W_\mathscr{L}$ with a pure QMC approach, we would therefore need a $\mathscr{N}_\mathscr{L}$-dimensional QMC point sequence, that we might not be able to sample given the restrictions of some modern QMC point generators\footnote{E.g.~currently there is no software support for good quality Sobol' sequences of dimension higher than 65536 \cite{Sobol2011}.}. In the interest of reducing the QMC dimension, we reorder the terms in the expression for $\W_\mathscr{L}$ in \eqref{eq:wavelet_QMC_white_noise} with respect to a total degree ordering rather than a full tensor grid ordering (i.e.~we reorder them with respect to the $1$ norm of $\bm{l}$ rather than the $\max$ norm) and we then only sample the Haar coefficients with $||\bm{l}||_1\leq \mathscr{L}$ which are always much less. To fix the notation, for $s\in\bar{\mathbb{N}}$, we define the set
\begin{align}
\mathcal{H}(s):=\{\bm{l} \in \bar{\mathbb{N}}^d\ :\ ||\bm{l}||_1=s,\ |\bm{l}|\leq \mathscr{L}\},\quad\text{with}\quad \mathcal{H}(-1) = \{\bm{-1}\}.
\end{align}
This is the set of all Haar level vectors in the expansion for $\W_\mathscr{L}$ of a given total degree. After reordering the terms in the expression for $\W_\mathscr{L}$ in \eqref{eq:wavelet_QMC_white_noise} with respect to the $1$-norm of $\bm{l}$ we obtain:
\begin{align}
\label{eq:WL_total_degree}
\W_\mathscr{L} = \sum\limits_{s=-1}^{\mathscr{L}} \sum\limits_{\bm{l}\in\mathcal{H}(s)}\sum\limits_{\bm{n}=\bm{0}}^{(2^{\bm{l}}-\bm{1})^+}z_{\bm{l},\bm{n}}(\omega)H_{\bm{l},\bm{n}}(\bm{x}) + \sum\limits_{s=\mathscr{L}+1}^{d\mathscr{L}} \sum\limits_{\bm{l}\in\mathcal{H}(s)}\sum\limits_{\bm{n}=\bm{0}}^{(2^{\bm{l}}-\bm{1})^+}z_{\bm{l},\bm{n}}(\omega)H_{\bm{l},\bm{n}}(\bm{x}),
\end{align}
where we sample the coefficients in the first sum on the right hand side with QMC and the remaining coefficients with a standard MC approach. We therefore adopt a hybrid sampling technique for $\W_\mathscr{L}$.

In order to achieve good convergence with respect to the number of QMC samples we order the QMC dimensions according to $s$ so that the first dimension corresponds to $z_{-\bm{1},\bm{0}}$, the second batch of dimensions correspond to the terms with $s=0$, the third batch of dimensions to the terms with $s=1$ and so on up until $s=\mathscr{L}$. We map each sampled QMC point (in our case Sobol' with digital shifting \cite{glasserman2013}) to a Gaussian-distributed sequence by applying the inverse Normal CDF. We sample the remaining coefficients independently using a pseudo-random number generator.

\begin{remark}
	Note that both orderings ($\max$ and $1$ norm) are essential: the white noise expansion \eqref{eq:Haar_expansion_full} must be split according to the max norm so that $\W_\mathscr{L}$ can be interpreted as the projection of white noise onto $V_H$. In principle, the max norm could also be used to enforce the ordering required for QMC convergence. However, this would involve sampling for an exceptionally high dimensional QMC point sequence. Ordering the terms in the expansion for $\W_\mathscr{L}$ according to the $1$ norm instead allows us to still enforce a good QMC ordering while reducing the dimension of the QMC point rule used.
\end{remark}

We now propose the following algorithm for sampling $\W_\mathscr{L}$:\\\vspace{-6pt}
\paragraph{Algorithm for the sampling of $\W_\mathscr{L}$}
\begin{enumerate}
	\item Compute the supermesh between the FE mesh and the Haar mesh to obtain $\dot{V}_h=\spn(\phi_1,\dots,\phi_m)$. Compute the scalar map $\kappa(i)$ that maps each $i$ to the index $k$ of the Haar cell $\square_k$ that contains the support of $\phi_i$ and compute $\int_D \phi_i \text{ d}\bm{x}$ for all $i=1,\dots,m$. This step can be done offline.
	\item Sample the vector $\bm{z}_\mathscr{L} \in \R^{\mathscr{N}_\mathscr{L}}$ of the coefficients in the expression \eqref{eq:wavelet_QMC_white_noise} for $\W_\mathscr{L}$ as $\bm{z}_\mathscr{L} = [\bm{z}_{\text{QMC}}^T,\bm{z}_{\text{MC}}^T]^T$, where $\bm{z}_{\text{QMC}}$ is a randomized QMC sequence point of dimension equal to the number of coefficients with $||\bm{l}||_1\leq \mathscr{L}$ and $\bm{z}_{\text{MC}}$ is sampled with a pseudo-random number generator.
	\item Loop over each Haar mesh cell $\square_k$ and sample the values $w_k$ of $\W_\mathscr{L}$ over $\square_k$ as follows. Let $J(\bm{l},\bm{n})$ be the index map that given $(\bm{l},\bm{n})$ returns the index $j$ such that ${z_{\bm{l}, \bm{n}} = (\bm{z}_\mathscr{L})_j}$ (the two vectors are the same up to reordering) and define $\bm{m}_k\in\R^{d}$ to be the coordinate vector of the midpoint of $\square_k$. For each $k=1,\dots,\mathscr{N}_\mathscr{L}$ and $\bm{l}$ with $|\bm{l}| \leq \mathscr{L}$, there is only one wavelet with level vector $\bm{l}$ with non-zero support over $\square_k$. For $i=1,\dots,d$, its wavelet number is given by $(\bar{\bm{n}}_k(\bm{l}))_i = \lfloor(\bm{m}_k)_i2^{\bm{l}_i}\rfloor$ and its sign over $\square_k$ by $\bar{s}_k(\bm{l}):=\prod_{i=1}^ds_k(\bm{l}_i)$, where the $s_k(\bm{l}_i)$ are the signs of the 1D Haar wavelets in the tensor product for $H_{\bm{l},\bar{\bm{n}}_k(\bm{l})}$, namely
	\begin{align}
	\label{alg:sign_1D_Haar_wavelet}
	s_k(\bm{l}_i) := 1-2(\lfloor(\bm{m}_k)_i2^{\bm{l}_i+1}\rfloor \pmod 2).
	\end{align}
	This expression comes from the fact that Haar wavelets are positive on even Haar cells and negative on odd cells. We set for all $k=1,\dots,\mathscr{N}_\mathscr{L}$,
	\begin{align}
	w_k := \sum\limits_{\bm{l}=\bm{0}}^{|\bm{l}|\leq \mathscr{L}}\bar{s}_k(\bm{l}) \bm{z}_{J(\bm{l},\bar{\bm{n}}_k(\bm{l}))}2^{||\bm{l}^+||_1/2}
	\end{align}
	\item For all $i=1,\dots,m$, set $(\bm{b}_\mathscr{L})_i := w_{\kappa(i)}\int_D\phi_i \text{ d}\bm{x}$ (note that here the support of each $\phi_i$ is entirely contained by the Haar cell $\square_{\kappa(i)}$). Assemble all these contributions into $\bm{b}_\mathscr{L}$.
\end{enumerate}

\begin{remark}
	\label{rem:intersections_with_haar_mesh}
	In point $1$ and $3$ above we exploit the fact that the Haar mesh is uniform and structured.
	For instance, we can readily obtain the Haar mesh cell in which any point $\bm{p}\in D$ lies: it belongs to the ${\lfloor(\bm{p})_i2^{\mathscr{L}+1}\rfloor\text{-th}}$ Haar cell from the origin in the $i$-th coordinate direction. The expressions for $\bar{\bm{n}}_k(\bm{l})$ and $\bar{s}_k(\bm{l})$ in point $3$ above also derive from the same considerations.
\end{remark}

\begin{remark}[Complexity of the sampling of $\W_\mathscr{L}$]
	\label{rem:supermesh_complexity}
	Let $m$ be the number of basis functions that span $\dot{V}_h$, let $\mathscr{N}_\mathscr{L}=2^{d(\mathscr{L}+1)}$ be the number of cells in the Haar mesh and let $N_\mathscr{L}=(\mathscr{L}+2)^d$ be the number of wavelets that are non-zero over a given Haar cell. In general, it is possible to sample $\W_\mathscr{L}$ in $O(m + N_\mathscr{L}\mathscr{N}_\mathscr{L})$ complexity, which reduces to $O(m + \mathscr{N}_\mathscr{L})$ in the case in which we are using non-standard Haar wavelets\footnote{In this case it is possible to use a multi-dimensional generalization of the Brownian bridge construction (of which $\W_\mathscr{L}$ is the derivative) which is well known in the computational finance literature \cite{glasserman2013}.} (cf.~remark \ref{rem:compactly_supported_Haar_basis}). If $D_h$ is not nested within $D_\mathscr{L}$ and a supermesh construction is needed the cost complexity becomes $O(\mathscr{N}_Sm_e + N_\mathscr{L}\mathscr{N}_\mathscr{L})$, where $\mathscr{N}_S$ is the number of cells in the supermesh between $D_h$ and $D_\mathscr{L}$ and $m_e$ is the number of degrees of freedom on each supermesh cell $e$. Owing to theorem 1.1 in \cite{quasi-uniform-supermeshing}, when $D_h$ is quasi-uniform we have $\mathscr{N}_S = O(\mathscr{N}_h + \mathscr{N}_\mathscr{L})$, where $\mathscr{N}_h$ is the number of cells of $D_h$ and $a > 0$. This gives a linear cost complexity in the number of cells of $D_h$ and log-linear in the number of cells of $D_\mathscr{L}$ since $N_\mathscr{L} = O((\log_2(\mathscr{N}_\mathscr{L})/d)^d)$. The log-term can be dropped if we use non-standard wavelets. 
\end{remark}

\subsection{Sampling of $\W_R$}

We now consider the efficient sampling of $\W_R$. Dealing with an infinite summation is complicated. However, we can circumvent this problem by noting that the covariance of $\W_R$ is known since, as $\W_\mathscr{L}$ is independent from $\W_R$ by construction, for all $u,v\in L^2(D)$ we have
\begin{align}
\label{eq:cov_Wr_splitting}
\E[\langle \W_R, u\rangle\langle \W_R, v\rangle] = \E[\langle \W, u\rangle\langle \W, v\rangle] - \E[(\W_\mathscr{L}, u)(\W_\mathscr{L}, v)],
\end{align}
where the covariance of $\W$ is known by definition \ref{def:white_noise} and the covariance of $\W_\mathscr{L}$ is given by the following lemma.
\begin{lemma}
	\label{lemma:covariance_Wr}
	Let $\square_i$ for $i=1,\dots,\mathscr{N}_\mathscr{L}$ be the $i$-th cell of $D_\mathscr{L}$ of volume\\ ${|\square_i| = 2^{-d(\mathscr{L}+1)}=|\square_H|}$ for all $i$ (see definition \ref{def:Haar_mesh}). Then, for all $u,v\in L^2(D)$,
	\begin{align}
	\mathcal{C}_\mathscr{L}(u,v) := \E[(\W_\mathscr{L}, u)(\W_\mathscr{L}, v)] = \sum\limits_{i=1}^{\mathscr{N}_\mathscr{L}}\frac{1}{|\square_i|}\int_{\square_i}u\text{ d}\bm{x}\int_{\square_i}v\text{ d}\bm{x}.
	\end{align}
\end{lemma}
\begin{proof}
	Let $P_H$ be the $L^2$ projection onto $V_H$, then for all $u\in L^2(D)$ we have that $P_Hu=\sum_{i=1}^{\mathscr{N}_\mathscr{L}}u_i\psi_i$ satisfies
	\begin{align}
	(P_Hu,v_H) = (u,v_H),\quad \forall v_H\in V_H.
	\end{align}
	A standard FE calculation gives that the coefficients $u_i$ are given by
	\begin{align}
	\label{eq:reproof1}
	u_i = \frac{1}{|\square_i|}(u,\psi_i) = \frac{1}{|\square_i|}\int_{\square_i}u \text{ d}\bm{x}.
	\end{align}
	
	We conclude by using lemma \ref{lemma:W_L_is_projection} to show that, for all $u,v\in V$ such that $P_Hu = \sum_{i=1}^{\mathscr{N}_\mathscr{L}}u_i\psi_i$ and $P_Hv = \sum_{i=1}^{\mathscr{N}_\mathscr{L}}v_i\psi_i$,
	\begin{align}
	\E[(\W_\mathscr{L},u)(\W_\mathscr{L},v)] = \E[\langle\W,P_Hu\rangle\langle\W,P_Hv\rangle] = (P_Hu,P_Hv) \notag\\
	= \sum\limits_{i,j=1}^{\mathscr{N}_\mathscr{L}}u_iv_j(\psi_i,\psi_j) = \sum\limits_{i=1}^{\mathscr{N}_\mathscr{L}}|\square_i|u_iv_i= \sum\limits_{i=1}^{\mathscr{N}_\mathscr{L}}\frac{1}{|\square_i|}\int_{\square_i}u\text{ d}\bm{x}\int_{\square_i}v\text{ d}\bm{x}.
	\end{align}
\end{proof}

\begin{remark}
	The sampling strategies for $\W_\mathscr{L}$ and $\W_R$ presented in this work are conceptually different. In the $\W_\mathscr{L}$ case we use the Haar wavelet representation to make sure that the variables in the quasi-random sequence are ordered correctly. Therefore the use of the Haar representation is crucial in the sampling of $\W_\mathscr{L}$. In the $\W_R$ case, instead, the ordering is irrelevant as $\W_R$ is sampled by using pseudo-random numbers. For this reason we can ``forget'' about the wavelet representation in this case and sample $\W_R$ as it is done for any standard Gaussian field, i.e.~by factorising its covariance matrix after discretization.
\end{remark}

It is then readily shown from lemma \ref{lemma:covariance_Wr} and from \eqref{eq:cov_Wr_splitting} that the covariance of $\W_R$ is
\begin{align}
\label{eq:C_R}
\mathcal{C}_R(u,v) := \E[\langle \W_R, u\rangle\langle \W_R, v\rangle] = (u,v) - \sum\limits_{i=1}^{\mathscr{N}_\mathscr{L}}\frac{1}{|\square_i|}\int_{\square_i}u\text{ d}\bm{x}\int_{\square_i}v\text{ d}\bm{x},
\end{align}
for all $u,v\in L^2(D)$.
Before proceeding, we show that $\mathcal{C}_R$ is a proper covariance function, i.e.~that it is positive semi-definite.
\begin{lemma}
	\label{lemma:covariance_W_R_pos_semidef}
	The covariance of $\W_R$, $\mathcal{C}_R$, is positive semi-definite.
\end{lemma}
\begin{proof}
	With the same notation as in the proof of lemma \ref{lemma:covariance_Wr}, we have that, for all $u\in L^2(D)$,
	\begin{align}
	\mathcal{C}_R(u,u) := \E[(\langle\W_R, u\rangle)^2] = \E[(\langle\W - P_H\W, u\rangle)^2] = \E[(\langle\W, u - P_Hu\rangle)^2] = ||u-P_Hu||_{L^2(D)}^2,
	\end{align}
	since $\W_R = \W - \W_\mathscr{L} = \W - P_H\W$. Hence $\mathcal{C}_R(u,u)$ is always non-negative and it is zero if and only if $u\in V_H$.
\end{proof}

\begin{remark}
	In principle, if $D_\mathscr{L}$ is fine enough so that $\W_{\mathscr{L}}\approx \W$, then the correction $\W_R$ is not needed at all. However, Haar wavelets are only piecewise constant and we might only expect first order convergence of $\W_\mathscr{L}$ to $\W$. If so, large QMC dimensions and a very fine Haar mesh would be needed to make the correction term $\W_R$ negligible and this translates into very expensive samples of $\W_\mathscr{L}$. If we also compute samples of $\W_R$, however, the Haar level can be kept small.
\end{remark}

The sampling of $\W_R$ can be performed independently on each Haar cell. The idea is to first sample the action of $\W_R$ onto the basis functions of $\dot{V}_h$. If we focus our attention only on the basis functions $\phi_1,\dots,\phi_{m_k}\in \dot{V}_h$ of support entirely contained within a given Haar cell $\square_k$ (i.e.~the basis functions of $V_h^b|_{\square_k}$), we note that the expression \eqref{eq:C_R} for $\mathcal{C}_R$ simplifies to
\begin{align}
\mathcal{C}_R(\phi_i,\phi_j) = (\phi_i,\phi_j) - \frac{1}{|\square_k|}\int_{\square_k}\phi_i\text{ d}\bm{x}\int_{\square_k}\phi_j\text{ d}\bm{x},\quad\text{for all }i,j\in\{1,\dots,m_k\}.
\end{align}
Similarly as in \cite{Croci2018}, the sampling of $\W_R$ over $\square_k$ boils down to sampling a zero-mean Gaussian vector $\bm{b}_R^k$ with entries $(\bm{b}_R^k)_i = \langle\W_R,\phi_i\rangle$ and covariance matrix $C_R^k$ of entries $(C_R^k)_{ij}$ given by
\begin{align}
\bm{b}_R^k\sim\mathcal{N}(0, C_R^k),\quad (C_R^k)_{ij} = \mathcal{C}_R(\phi_i,\phi_j).
\end{align}
If we let $M_k$ be the local mass matrix over the space spanned by the $\{\phi_i\}_{i=1}^{m_k}$, with entries $(M_k)_{ij}=(\phi_i,\phi_j)$ and if we let the vector $\bm{I}^k\in\R^{m_k}$ be given by
\begin{align}
\label{eq:Ik}
\bm{I}^k = \left[\int_{\square_k}\phi_1\text{ d}\bm{x},\dots,\int_{\square_k}\phi_{m_k}\text{ d}\bm{x}\right]^T,
\end{align}
we can write $C_R^k$ as \vspace{-9pt}
\begin{align}
\label{eq:C_R^k}
C_R^k = M_k - \frac{1}{|\square_k|}\bm{I}^k(\bm{I}^k)^T.
\end{align}
The sampling of a Gaussian vector with this covariance through factorization is expensive as direct factorization of $C_R^k$ (e.g.~Cholesky) has an $O(m_k^3)$ and $O(m_k^2)$ cost and memory complexity respectively and it is therefore to be avoided.

We now show how $\bm{b}_R^k$ can be sampled efficiently by extending the techniques presented in \cite{Croci2018}. The main idea is to first sample a Gaussian vector with covariance $M_k$ in linear complexity and then perform an efficient update to obtain a sample of $\bm{b}_R^k$. We can write the action of $\W_R$ against each $\phi_i$ as
\begin{align}
\label{eq:W_R_against_phi_i}
\langle\W_R,\phi_i\rangle = \langle\W-\W_\mathscr{L},\phi_i\rangle = \langle\W,\phi_i\rangle - \langle\W_\mathscr{L},\phi_i\rangle = (\bm{b}_M^k)_i - w_k(\bm{I}^k)_i,
\end{align}
where $\bm{I}^k$ is given by \eqref{eq:Ik}, $w_k$ by 
\begin{align}
\label{eq:def_of_wk}
w_k = \frac{1}{|\square_k|}\langle\W,\ind_{\square_k}\rangle,\quad w_k\sim\mathcal{N}\left(0,\frac{1}{|\square_k|}\right),
\end{align}
and the vector $\bm{b}_M^k\in\R^{m_k}$ is given entrywise by
\begin{align}
(\bm{b}^k_M)_i = \langle\W,\phi_i\rangle,\quad i=1,\dots,m_k.
\end{align}
The variables $w_k$ and $\bm{b}_M^k$ are by definition \ref{def:white_noise} all zero-mean joint Gaussian variables with covariance
\begin{align}
\E[w_kw_k] = \frac{1}{|\square_k|},\quad \E[\bm{b}^k_Mw_k] = \frac{\bm{I}^k}{|\square_k|},\quad \E[\bm{b}^k_M(\bm{b}^k_M)^T]=M_k.
\end{align}
Thanks to these relations and to \eqref{eq:W_R_against_phi_i}, if we set 
\begin{align}
\bm{b}_R^k := \bm{b}_M^k - w_k\bm{I}^k,
\end{align}
then the covariance of $\bm{b}_R^k$ is correct (cf. equation \eqref{eq:C_R^k}) since
\begin{align}
\E[\bm{b}_R^k(\bm{b}_R^k)^T] &= \E[(\bm{b}_M^k - w_k\bm{I}^k)(\bm{b}_M^k - w_k\bm{I}^k)^T] \notag\\
&=\E[\bm{b}^k_M(\bm{b}^k_M)^T] - \E[\bm{b}^k_Mw_k](\bm{I}^k)^T - \bm{I}^k\E[\bm{b}^k_Mw_k]^T + \E[w_kw_k]\bm{I}^k(\bm{I}^k)^T \notag\\
&= M_k - \frac{1}{|\square_k|}\bm{I}^k(\bm{I}^k)^T.
\end{align}
In what follows, we exploit the fact that constants can be represented exactly by the FE subspace $\dot{V}_h$, i.e.~$c\in \dot{V}_h$ for all $c\in\R$. Let $\bm{\phi}_k = [\phi_1,\dots,\phi_{m_k}]^T$. This means that for each Haar cell $\square_k$ there exists a vector $\bm{c}_k\in\R^{m_k}$ such that $\ind_{\square_k}\equiv\bm{c}_k\cdot \bm{\phi}_k$. It is then straightforward to obtain $w_k$ from $\bm{b}^k_M$ since
\begin{align}
\label{eq:derive_wk_from_b_M}
\bm{c}_k\cdot\bm{b}^k_M = \sum\limits_{i=1}^{m_k}\langle \W, (\bm{c}_k)_i\phi_i\rangle = 
\langle \W, \bm{c}_k\cdot \bm{\phi}_k\rangle = \langle \W, \ind_{\square_k}\rangle = |\square_k|w_k,
\end{align}
hence $w_k = |\square_k|^{-1}\bm{c}_k\cdot\bm{b}^k_M$. Note that $\bm{c}_k$ is always known, e.g.~for Lagrange basis functions on simplices we have $\bm{c}_k=\bm{1}\in\R^{m_k}$.

We are now ready to sample $\W_R$ from its distribution. The idea is to use the strategy presented in \cite{Croci2018} and summarized in section \ref{subsec:fast_WN_sampling} to sample the $L^2$ projection of white noise over $\dot{V}_h$, and then to orthogonalize the result against $V_H$. This operation can be performed Haar-cellwise as follows:\\\vspace{-6pt}

\paragraph{Algorithm for the efficient sampling of $\W_R$}
\begin{enumerate}
	\item Loop over each Haar cell $\square_k$.
	\item Sample a Gaussian vector $\bm{b}^k_M\sim\mathcal{N}(0,M_k)$ in linear cost complexity using the technique presented in section 4.1 of \cite{Croci2018} (recall section \ref{subsec:fast_WN_sampling}). This operation is equivalent to sampling white noise over $V_h^b|_{\square_k}$.
	\item Set $w_k := |\square_k|^{-1}\bm{c}_k\cdot\bm{b}^k_M$ and compute $\bm{b}_R^k := \bm{b}_M^k - w_k\bm{I}^k$. If $\square_k$ is on the boundary of $D$, remove the entries of $\bm{b}_R^k$ corresponding to basis functions of $\dot{V}_h$ which do not vanish on $\partial D$ (recall how we constructed $\dot{V}_h$). 
\end{enumerate}

\begin{remark}
	This algorithm has $O(\mathscr{N}_Sm_e^3)$ cost and $O(\mathscr{N}_Sm_e^2)$ memory complexity, where $\mathscr{N}_S$ is the total number of supermesh cells \cite{Croci2018}. As discussed in remark \ref{rem:supermesh_complexity}, $\mathscr{N}_S$ is of $O(\mathscr{N}_h + \mathscr{N}_\mathscr{L})$, where $\mathscr{N}_h$ and $\mathscr{N}_\mathscr{L}$ are the number of cells of $D_h$ and of $D_\mathscr{L}$ respectively.
\end{remark}

\section{Sampling coupled realizations for MLQMC}
\label{sec:coupled_realizations_MLQMC}

We now generalize the QMC sampling algorithm just presented to the MLQMC case. Compared to standard Monte Carlo, both MLMC and QMC already bring a significant computational improvement. When the two are combined into MLQMC, it is sometimes possible to obtain the best of both worlds and further improve the computational complexity and speed. However, to do so, we must be able to satisfy the requirements and assumptions underlying both QMC and MLMC: we must order the dimensions of our random input in decaying order of importance as in QMC and introduce an approximation level hierarchy and enforce a good coupling between the levels as in MLMC. We now show how this can be done with white noise sampling.

In what follows we assume we have a MLQMC hierarchy of possibly non-nested FE approximation subspaces $\{V^\ell\}_{\ell=1}^L$ over the meshes $\{D_h^\ell\}_{\ell=1}^L$ and of accuracy increasing with $\ell$. Since as in the MLMC case (see \cite{Croci2018}) the only stochastic element in \eqref{eq:white_noise_SPDE_reminder} is white noise, on each MLQMC level we must be able to draw Mat\'ern field samples $u_\ell\in V^\ell$ and $u_{\ell-1}\in V^{\ell-1}$ for $\ell > 1$ that satisfy the following variational problems coupled by the same white noise sample: for a given $\omega_\ell^n\in\Omega$, find $u_\ell\in V^\ell$ and $u_{\ell-1}\in V^{\ell-1}$ such that
\begin{align}
\label{eq:coupled1}
(u_{\ell},v_{\ell}) + \kappa^{-2}(\nabla u_{\ell},\nabla v_{\ell})         &= \langle \W,v_{\ell} \rangle (\omega^n_\ell),\quad\hspace{10.5pt}\text{for all } v_{\ell}\in V^\ell,\\
(u_{\ell-1},v_{\ell-1}) + \kappa^{-2}(\nabla u_{\ell-1},\nabla v_{\ell-1}) &= \langle \W,v_{\ell-1} \rangle (\omega^n_\ell),\quad\text{for all } v_{\ell-1}\in V^{\ell-1}.
\label{eq:coupled2}
\end{align}
where the terms on the right hand side are coupled in the sense that they are centred Gaussian random variables with covariance $\E[\langle \W,v_{\lell} \rangle \langle \W,v_{s} \rangle] = (v_{\lell},v_s)$ for $\lell,s\in\{\ell,\ell-1\}$, as given by definition \ref{def:white_noise}. Again we order the dimensions of white noise by expanding it in the Haar wavelet basis as in \eqref{eq:Haar_expansion_full}, but this time we allow the Haar level to possibly increase with the MLQMC level and we split the expansion at the finer Haar level between the two MLQMC levels, $\mathscr{L}_\ell$,
\begin{align}
\label{eq:wavelet_MLQMC_white_noise1}
\W = \W_{\mathscr{L}_\ell} + \W_{R_\ell},
\end{align}
where the splitting of the expansion is done in the same way as in equation \eqref{eq:wavelet_QMC_white_noise}. The choice of Haar levels $\mathscr{L}_\ell$ on each MLQMC level is problem-dependent. We discuss this further in section \ref{sec:MLQMC_num_res}. From now on we assume that $\mathscr{L}_{\ell-1}\leq \mathscr{L}_\ell$, although extending the methods presented to decreasing Haar level hierarchies is straightforward. Note that the splitting of the expansion at the MLQMC level $\ell$ is done at Haar level $\mathscr{L}_\ell$ independently from the value of $\mathscr{L}_{\ell-1}$, i.e.~$\W$ on the coarser grid is always sampled with Haar level $\mathscr{L}_{\ell}$ on MLQMC level $\ell$ and not with Haar level $\mathscr{L}_{\ell-1}$. However, this does not affect the MLQMC telescoping sum since $\W$ is always sampled without bias on all levels owing to the correction term $\W_{R_\ell}$.

Let $\{\varphi^\ell_i\}_{i=1}^{m_\ell}$ and $\{\varphi^{\ell-1}_j\}_{j=1}^{m_{\ell-1}}$ be the basis functions spanning $V^\ell$ and $V^{\ell-1}$ respectively. Sampling white noise on both MLQMC levels again means to sample the vectors $\bm{b}_\mathscr{L}^\ell$, $\bm{b}_\mathscr{L}^{\ell-1}$, $\bm{b}_R^\ell$ and $\bm{b}_R^{\ell-1}$, with entries given by,
\begin{align}
(\bm{b}_\mathscr{L}^{\lell})_i := \langle \W_{\mathscr{L}_\ell}, \varphi^{\lell}_i\rangle,\quad (\bm{b}_R^{\lell})_i := \langle \W_{R_\ell}, \varphi^{\lell}_i\rangle,\quad\text{for } i = 1,\dots,m_{\lell},\quad \lell\in\{\ell,\ell-1\}.
\end{align}

Since we both require a multilevel coupling and a Haar wavelet expansion, this time we need to construct a \emph{three-way} supermesh $S_h$ between $D_{\mathscr{L}_\ell}$, $D_h^\ell$ and $D_h^{\ell-1}$ (note that $D_{\mathscr{L}_{\ell-1}}$ is always nested within $D_{\mathscr{L}_\ell}$ so there is no need for a four-way supermesh). Thanks to the supermesh construction we can split the support of all the basis functions of $V^\ell$ and $V^{\ell-1}$ to obtain the augmented subspaces $\dot{V}^\ell$ and $\dot{V}^{\ell-1}$ such that each $\phi^\ell_i\in\dot{V}^\ell$ and $\phi^{\ell-1}_j\in\dot{V}^{\ell-1}$ has support entirely contained within a single Haar cell. This is done in the same way as we did in the QMC case (cf.~section \ref{sec:QMC_sample_white_noise}) so that $V^{\lell},D_h^{\lell}\subseteq \dot{V}^{\lell} \subseteq L^2(D)$ for $\lell\in\{\ell,\ell-1\}$. The sampling of $\W$ in the MLQMC case is extremely similar to that of the QMC case with only a few differences concerning the sampling of $\W_{R_\ell}$ which we will now highlight.

Again, portions of $\W_{R_\ell}$ on separate Haar cells of $D_{\mathscr{L}_\ell}$ are independent and we can therefore sample $\W_{R_\ell}$ Haar cell-wise. As in the QMC case, we first sample $\W_{R_\ell}$ on the augmented subspaces $\{\dot{V}^\ell\}_{\ell=1}^L$. For each Haar cell $\square_k$ and for $\lell\in\{\ell,\ell-1\}$, let $\phi^{\lell}_1,\dots,\phi^{\lell}_{m^{\lell}_k}$ be the basis functions of $\dot{V}^{\lell}$ with non-zero support over $\square_k$ and define the Haar cell correction vectors $\bm{b}_{R,k}^{\lell}$ with entries $(\bm{b}_{R,k}^{\lell})_i = \langle \W, \phi^{\lell}_i\rangle$ for $i\in\{1,\dots,m^{\lell}_k\}$ and covariances given by,
\begin{align}
\E[\bm{b}_{R,k}^{\lell}(\bm{b}_{R,k}^{\lell})^T] = M_k^{\lell} - \frac{1}{|\square_k|}\bm{I}^k_{\lell}(\bm{I}^k_{\lell})^T,\quad \E[\bm{b}_{R,k}^\ell(\bm{b}_{R,k}^{\ell-1})^T]=M_k^{\ell,\ell-1} - \frac{1}{|\square_k|}\bm{I}^k_\ell(\bm{I}^k_{\ell-1})^T,
\end{align}
where $(M_k^{\lell})_{ij}=(\phi^{\lell}_i,\phi^{\lell}_j)$, $(M_k^{\ell,\ell-1})_{ij}=(\phi^\ell_i,\phi^{\ell-1}_j)$ and $(\bm{I}^k_{\lell})_i = \int_D \phi^{\lell}_i \text{ d}\bm{x}$. If we define $w_k$ as in \eqref{eq:def_of_wk} we can again write
\begin{align}
\bm{b}_{R,k}^{\lell} = \bm{b}_{M,k}^{\lell} - w_k\bm{I}^k_{\lell},\quad\text{for }\lell\in\{\ell,\ell-1\},
\end{align}
where $\bm{b}_{M,k}^{\lell} \sim\mathcal{N}(0,M_k^{\lell})$. Since constants can be represented exactly by both $\dot{V}^\ell$ and $\dot{V}^{\ell-1}$, i.e.~for all $c\in\R$ and for all $\lell\in\{\ell,\ell-1\}$ we have that $c\in \dot{V}^{\lell}$, then there exist two vectors $\bm{c}_k^\ell$ and $\bm{c}_k^{\ell-1}$ such that $\ind_{\square_k} \equiv \bm{c}_k^\ell\cdot\bm{\phi}_k^\ell\equiv\bm{c}_k^{\ell-1}\cdot\bm{\phi}_k^{\ell-1}$, where $\bm{\phi}_k^{\lell} = [\phi_1^{\lell},\dots,\phi_{m_k^{\lell}}^{\lell}]^T$ for $\lell\in\{\ell,\ell-1\}$. The same argument used to derive equation \eqref{eq:derive_wk_from_b_M} then gives
\begin{align}
w_k = \frac{1}{|\square_k|}\bm{c}_k^\ell\cdot\bm{b}_{M,k}^\ell = \frac{1}{|\square_k|}\bm{c}_k^{\ell-1}\cdot\bm{b}_{M,k}^{\ell-1}.
\end{align}
We can now proceed with the coupled sampling of $\W$ for MLQMC as follows:\\\vspace{-6pt}

\paragraph{Algorithm for the efficient sampling of $\W$ for MLQMC}
\begin{enumerate}
	\item Compute the three-way supermesh between the FE meshes and the Haar mesh $D_{\mathscr{L}_\ell}$ and construct $\dot{V}^\ell$ and $\dot{V}^{\ell-1}$. Compute the scalar maps $\kappa^{\lell}(i)$ that map each $i$ to the index $k$ of the Haar cell $\square_k$ that contains the support of $\phi_i^{\lell}$ and compute $\int_D \phi_i^{\lell} \text{ d}\bm{x}$ for all $i=1,\dots,m^l$ and for ${\lell}\in\{\ell,\ell-1\}$. This step can be done offline.
	\item Let $\mathscr{N}_{\mathscr{L}_\ell}$ be the number of cells of $D_{\mathscr{L}_\ell}$. Sample the vector $\bm{z}_{\mathscr{L}_\ell} \in \R^{\mathscr{N}_{\mathscr{L}_\ell}}$ of the coefficients in the expression for $\W_{\mathscr{L}_\ell}$ as $\bm{z}_\mathscr{L} = [\bm{z}_{\text{QMC}}^T,\bm{z}_{\text{MC}}^T]^T$, where $\bm{z}_{\text{QMC}}$ is a randomized QMC sequence point of dimension equal to the number of coefficients with $||\bm{l}||_1\leq {\mathscr{L}_\ell}$ and $\bm{z}_{\text{MC}}$ is sampled with a pseudo-random number generator.
	\item Compute the Haar cell values $\bar{w}_k$ of $\W_\mathscr{L}$ over all $\square_k$ for $k=1,\dots,\mathscr{N}_{\mathscr{L}_\ell}$ in the same way as in the QMC case (this step does not depend on the FE meshes).
	\item Use the technique presented in section 4.2 of \cite{Croci2018} to work supermesh cell by supermesh cell and sample in linear cost complexity the coupled Gaussian vectors $\bm{b}^\ell_{M,k}$ and $\bm{b}^{\ell-1}_{M,k}$ with covariance,
	\begin{align}
	\E\left[
	\left[\begin{array}{c}
	\bm{b}^\ell_{M,k}\\ \hline
	\bm{b}^{\ell-1}_{M,k}
	\end{array}\right]
	\left[\begin{array}{c}
	\bm{b}^\ell_{M,k}\\ \hline
	\bm{b}^{\ell-1}_{M,k}
	\end{array}\right]^T
	\right]
	= 
	\left[\begin{array}{c|c}
	M_k^\ell & M_k^{\ell,\ell-1} \\ \hline
	(M_k^{\ell,\ell-1})^T & M_k^{\ell-1}
	\end{array}\right].
	\end{align}
	This operation is equivalent to sampling the $L^2$ projection of white noise over $V_{S_\ell}|_{\square_k}=\dot{V}^\ell|_{\square_k} + \dot{V}_{\ell-1}|_{\square_k}$ using the technique from section \ref{subsec:fast_WN_sampling}. See \cite{Croci2018} for further details.
	\item For all $\lell\in\{\ell,\ell-1\}$, compute $(\bm{b}_\mathscr{L}^{\lell})_i := \bar{w}_{\kappa^{\lell}(i)}\int_D\phi_i^{\lell} \text{ d}\bm{x}$ for all $i=1,\dots,m^{\lell}$, then set $w_k := |\square_k|^{-1}\bm{c}_k^{\ell-1}\cdot\bm{b}_{M,k}^{\ell-1}$ and compute $\bm{b}_{R,k}^{\lell} = \bm{b}_{M,k}^{\lell} - w_k\bm{I}^k_{\lell}$. If $\square_k$ is on the boundary of $D$, remove the entries of $\bm{b}_{R,k}^{\lell}$ corresponding to the basis functions of $\dot{V}^{\lell}$ which do not vanish on $\partial D$ (cf.~section \ref{sec:QMC_sample_white_noise}).
\end{enumerate}

\begin{remark}[Complexity of the sampling of $\W$ for MLQMC]
	\label{rem:MLQMC_sampling_complexity}
	The overall complexity of this sampling strategy is $O(\mathscr{N}_Sm_e + N_{\mathscr{L}_\ell}\mathscr{N}_{\mathscr{L}_\ell})$ in the standard Haar wavelet case and $O(\mathscr{N}_{S_\ell}m_e^\ell + \mathscr{N}_{\mathscr{L}_\ell})$ in the non-standard case (cf.~remark \ref{rem:compactly_supported_Haar_basis}), where (cf.~remark \ref{rem:supermesh_complexity}) $\mathscr{N}_{S_\ell}$ is the number of cells of the three-way supermesh on the MLQMC level $\ell$, $m_e^\ell$ is the maximum number of dofs of $V^\ell$ per cell of $D_h^\ell$ and $N_{\mathscr{L}_\ell}$ is the number of wavelets that have non-zero support over any of the $\mathscr{N}_{\mathscr{L}_\ell}$ cells of $D_{\mathscr{L}_\ell}$. Since $\mathscr{N}_{S_\ell} = O(\mathscr{N}_h^\ell + \mathscr{N}_{\mathscr{L}_\ell})$ (cf.~theorem 1.1 in \cite{quasi-uniform-supermeshing}), where $\mathscr{N}_h^\ell$ is the number of cells of $D_h^\ell$, this gives an overall linear cost complexity in $\mathscr{N}_h^\ell$ and log-linear (linear for non-standard wavelets) in $\mathscr{N}_{\mathscr{L}_\ell}$.
\end{remark}

\begin{remark}[Simpler cases: nested meshes and $p$-refinement]
	When the MLQMC mesh hierarchy is nested and/or the hierarchy is composed by taking a single mesh and increasing the polynomial degree of the FE subspaces we have $V^{\ell-1}\subseteq V^\ell$. In this case everything we discussed still applies with the following simplifications: only a two-way supermesh between $D_h^\ell$ and $D_{\mathscr{L}_\ell}$ is needed on each MLQMC level in the $h$-refinement case. In the $p$-refinement case we only have one FE mesh $D_h$ and a single two-way supermesh construction is needed between $D_h$ and the finest Haar mesh $D_{\mathscr{L}_L}$.
\end{remark}

\begin{remark}[Non-nested mesh hierarchies and embedded domains]
	We assume that we are given a user-provided hierarchy  $\{G_h^\ell\}_{\ell=1}^{L}$ of possibly non-nested FE meshes of the domain $G$ on which we need the Mat\'ern field samples. From this we construct a boxed domain $D$ s.t.~$G\subset\joinrel\subset D$ and a corresponding hierarchy of Haar meshes $\{D_{\mathscr{L}_\ell}\}_{\ell=1}^{L}$ and of FE meshes of $D$, $\{D_h^\ell\}_{\ell=1}^{L}$. As in \cite{Croci2018}, it is convenient to construct each $D_h^\ell$ so that $G_h^\ell$ is nested within it, so that each Mat\'ern field sample can be transferred exactly and at negligible cost on the mesh on which it is needed (this is the embedded domain strategy proposed in \cite{Osborn2017}).
\end{remark}

\begin{remark}[Generic wavelets]
	\label{rem:QMC_generalizations_wavelets}
	It should be possible to generalize the presented sampling methods to generic orthogonal, compactly supported wavelets, although it is unclear as to whether this would bring any considerable advantage. This would likely increase the algorithm complexity and we leave this investigation to future research.
\end{remark}

\begin{remark}[General domains]
	\label{rem:QMC_generalizations}
	The same sampling method could possibly be generalized to arbitrary convex domains by introducing ``generalized'' Haar wavelets and meshes, obtained by partitioning $D_h$ into $\mathscr{N}_{\mathscr{L}}$ sub-regions each of which defining a Haar mesh cell. In this case the ``Haar cells'' would not be boxes, but polygons or polyhedra obtained by aggregating the FEM cells of the sub-regions and the ``Haar wavelets'' would just be piecewise constant over these cells. Establishing any theoretical results in this case would be more complex, but the same algorithm should carry forward after accounting for the fact that the ``Haar cells'' obtained through aggregation would have variable volume. The advantage of doing this is that no supermesh would then be required in the QMC case (the Haar mesh would be nested within $D_h$ by construction) and only one supermesh construction would be needed (between $D_h^\ell$ and $D_h^{\ell-1}$) in the non-nested MLQMC case. We leave the non-trivial implementation of this extension to future work.
\end{remark}

\section{Numerical results}
\label{sec:MLQMC_num_res}
We now test the algorithms presented. We consider problem \eqref{eq:diffusion_eqn_for_QMC_conv_general} over $G=(-0.5,0.5)^d$ with $F(x)=\exp(x)$ and forcing term $f = 1$, i.e.~we solve
\begin{align}
\label{eq:diffusion_eqn_for_QMC_conv_bis}
\begin{array}{rlc}
-\nabla\cdot(e^{u(\bm{x},\omega)}\nabla p(\bm{x},\omega)) = 1, & \bm{x}\in G =(-0.5,0.5)^d,& \omega\in\Omega,\\
p(\bm{x},\omega) = 0,& \bm{x}\in \partial G,& \omega\in\Omega,
\end{array}
\end{align}
where $u(\bm{x},\omega)$ is a Mat\'ern field sampled by solving equation \eqref{eq:white_noise_SPDE_reminder} over $D=(-1,1)^d$ with $\lambda = 0.125$ and mean and standard deviation chosen so that $\E[e^u]=1$, $\V[e^u]=0.2$. For simplicity, we take the squared $L^2(G)$ norm of $p$, $P(\omega) = ||p||_{L^2(G)}^2(\omega)$ as our output functional of interest.
\begin{remark}
	We do not consider functionals of the Mat\'ern field $u$ and we directly focus on the estimation of $\E[P]$. The reason is that in 2D and 3D the smoothness of $u$ is low and we only observe standard Monte Carlo convergence rates in numerical experiments (not shown).
\end{remark}

We solve equations \eqref{eq:white_noise_SPDE_reminder} and \eqref{eq:diffusion_eqn_for_QMC_conv_bis} with the FEniCS software package \cite{LoggEtAl2012}. For simplicity, we consider the $h$-refinement case and we discretize the equations using continuous piecewise-linear Lagrange elements. We employ the conjugate gradient routine of PETSc \cite{balay2014petsc} preconditioned by the BoomerAMG algebraic multigrid algorithm from Hypre \cite{hypre} to solve the linear systems arising from both equations.
We declare convergence when the absolute size of the preconditioned residual norm is
below a tolerance of $10^{-12}$. We employ the libsupermesh software package \cite{libsupermesh-tech-report} for the supermesh constructions. We use random digital shifted Sobol' sequences sampled with a custom-built\footnote{Available online at \url{bitbucket.org/croci/mkl_sobol/}.} Python and C wrapper of the Intel\textsuperscript{\tiny\textregistered} Math Kernel Library Sobol' sequence implementation augmented with Joe and Kuo's primitive polynomials and direction numbers \cite{JoeKuo2008} (maximum dimension $= 21201$). All the algorithms presented (as well as the MLMC methods from \cite{Croci2018}) are available online within the femlmc software package\footnote{Available online at \url{bitbucket.org/croci/femlmc/}.}.

We construct the mesh hierarchies $\{G_h^\ell\}_{\ell=1}^{L}$ and $\{D_h^\ell\}_{\ell=1}^{L}$ so that, for all MLQMC levels $\ell$, $G_h^\ell$ is nested within $D_h^\ell$, yet $G_h^{\ell-1}$ and $D_h^{\ell-1}$ are not nested within $G_h^{\ell}$ and $D_h^{\ell}$ respectively. We take all the meshes in both hierarchies to be simplicial, uniform and structured with mesh sizes $h_\ell = 2^{-1/2}\ 2^{-\ell}$ in 2D and $h_\ell = \sqrt{3}\ 2^{-(\ell + 1)}$ in 3D, although we do not exploit this structure in the implementation. The motivation behind this choice of grids is that these simpler test cases make for a cleaner asymptotic behaviour, making it easier to match convergence theory with computations\footnote{The use of unstructured non-nested grids can make the convergence behaviour erratic due to variation in mesh quality across the hierarchy, cf.~\cite{Croci2018}.}. For an application of the QMC sampling algorithm to more complicated grids and geometries, see \cite{CrociVinjeRognes2019bis}.

\begin{figure}[h!]
	\centering
	\begin{subfigure}{.35\textwidth}
		\includegraphics[width=\textwidth, trim={0cm 0cm 0cm 0cm},clip]{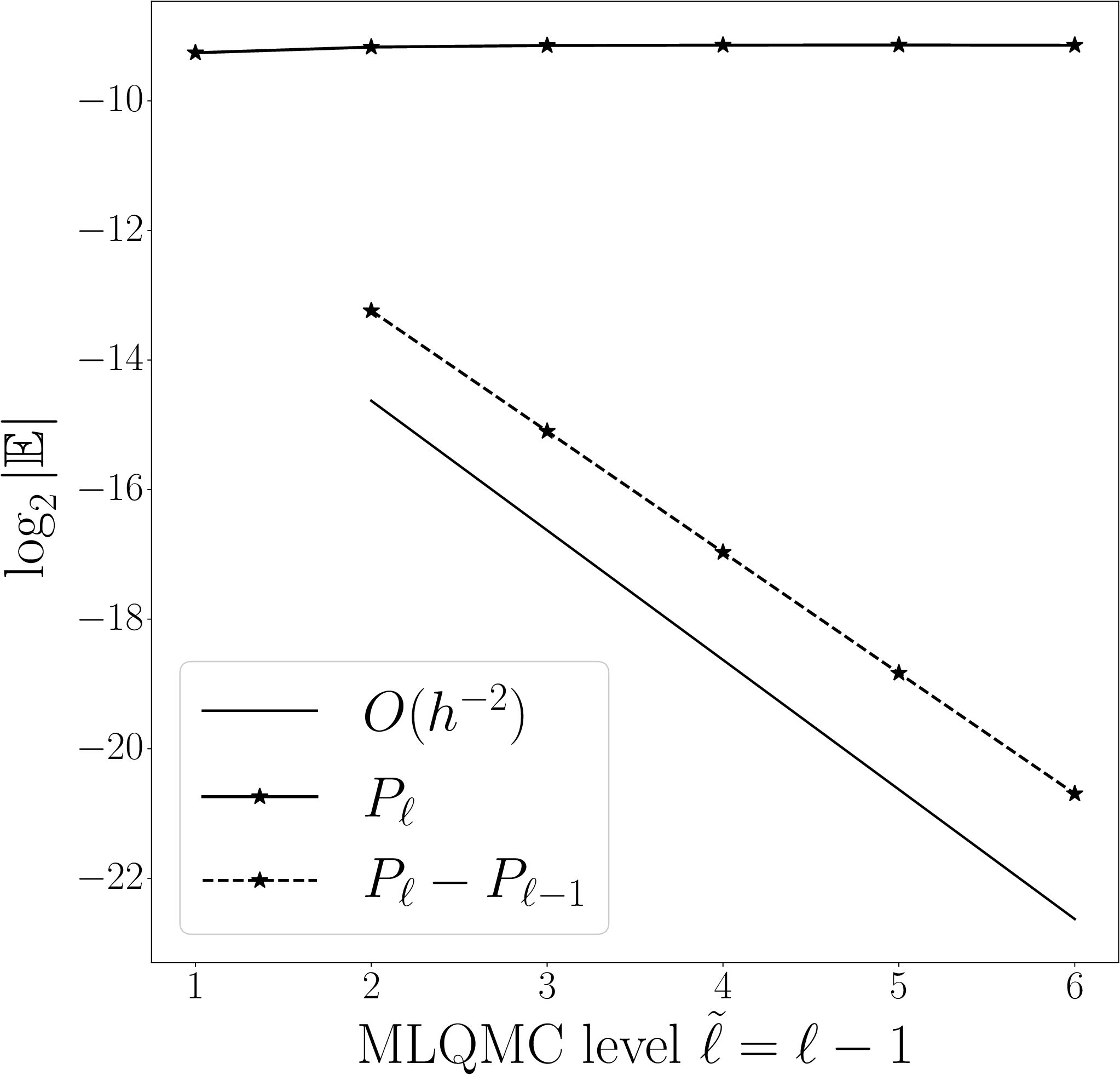}
	\end{subfigure}
	\vspace{6pt}
	\begin{subfigure}{.35\textwidth}
		\includegraphics[width=\textwidth, trim={0cm 0cm 0cm 0cm},clip]{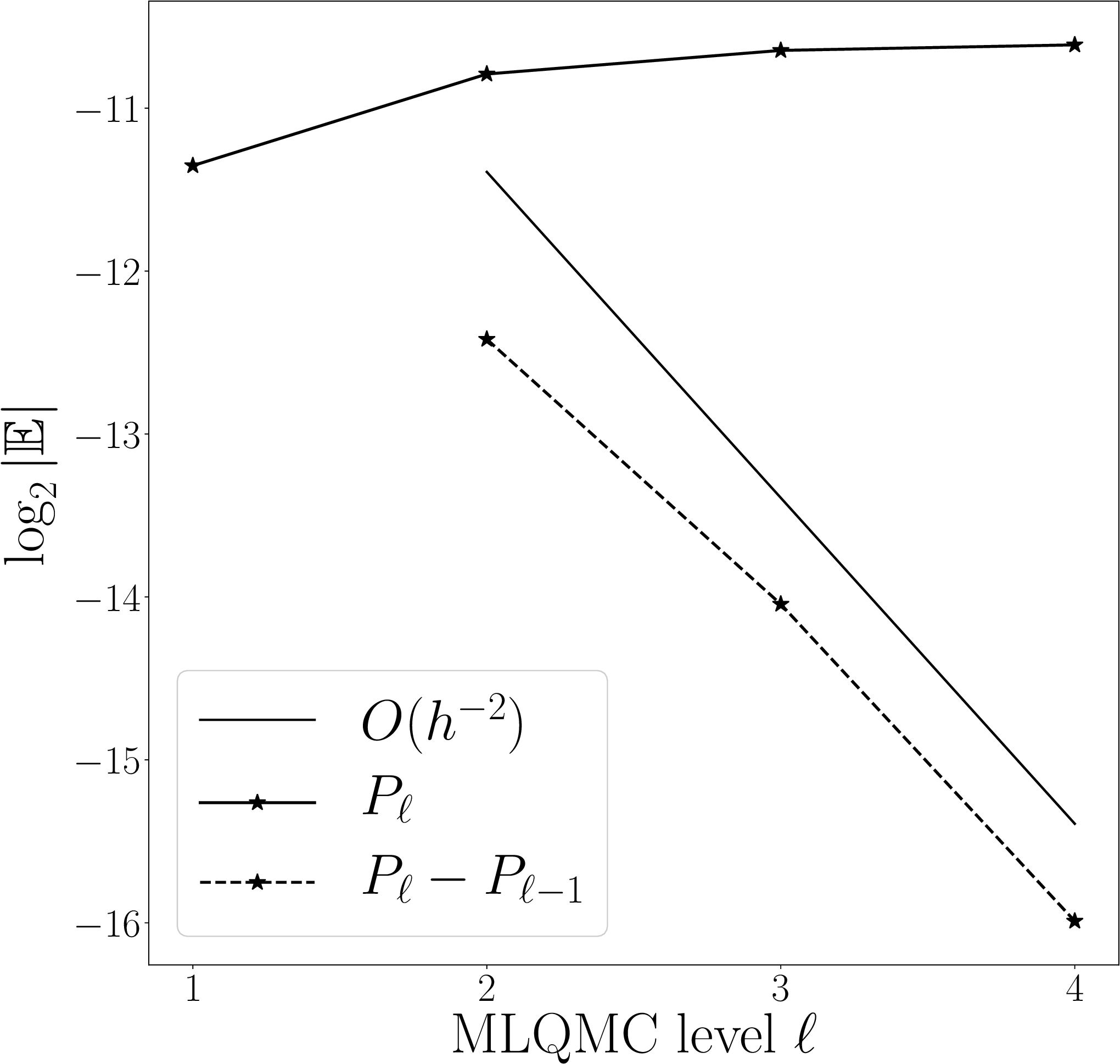}
	\end{subfigure}
	\vspace{-6pt}
	\caption{\textit{Logarithm of the absolute value of the expected value of $P_\ell$ and $P_\ell-P_{\ell-1}$ as a function of the MLQMC level $\ell$ in 2D (left) and 3D (right). We observe a decay rate of $O(h^{-2})$ in both dimensions. These results are independent from the Haar level chosen as we always compute the exact action of white noise independently from the choice of $\mathscr{L}$. In 2D the coarsest mesh of the hierarchy was dropped by the MLQMC routine as it is too coarse and it would not bring any significant performance advantage (same reasoning as for MLMC, see section 2.6 in \cite{giles2015multilevel}).}}
	\label{fig:MLQMC_bias_conv}
	\vspace{-18pt}
\end{figure}

We first study how the quantities $|\E[P_\ell]|$ and $|\E[P_\ell-P_{\ell-1}]|$ vary as the MLQMC level is increased. Assuming that $u$ can be sampled exactly, we expect the MLMC parameter value $\alpha$ to be $\alpha = \min(\nu+1, \bar{p}+1)$, where $\bar{p}$ is the polynomial degree \cite{Hackbusch1992}. Numerical results are shown in figure \ref{fig:MLQMC_bias_conv} where we observe a decay rate of $\alpha=2$ in 2D and 3D. In 3D we might have expected the rate to be $1.5$ due to the lack of smoothness of the coefficient $u$ which is only in $C^{0.5-\epsilon}(\bar{G})$ for any $\epsilon>0$ \cite{Hackbusch1992}. However, at the discrete level, the FEM approximation of $u$ is in $W^{1,\infty}(G)$ a.s.~and we might be observing a pre-asymptotic regime. 

As a next step, we analyse the convergence behaviour of QMC and MLQMC with respect to the number of samples. In the supplementary material (theorem \ref{th:hybrid_QMC_conv}) we show that in the QMC case we expect an initial faster-than-MC convergence rate followed by a standard MC rate of $O(N^{-1/2})$ and that the higher the Haar level is, the later the transition between the two regimes happens. No results regarding the MLQMC case were derived, but we expect a similar behaviour to occur. Furthermore, we would like to determine whether the multilevel technique can improve on QMC by bringing further variance reduction.

We draw inspiration from the original MLQMC paper by Giles and Waterhouse \cite{GilesWaterhouse2009} and we study the convergence behaviour of both QMC and MLQMC as the MLQMC level is increased. Results are shown in figure \ref{fig:QMC_conv}. We increase the Haar level with the MLQMC level so that the Haar mesh size is always proportional to the FE mesh size, but we consider two different strategies: 1) we choose the Haar mesh size to be comparable to the FE mesh size (figures \ref{fig:QMC_conv3}, \ref{fig:QMC_conv5}) and 2) we pick the Haar mesh size to be smaller than the FE mesh size (figures \ref{fig:QMC_conv4}, \ref{fig:QMC_conv6}). For both scenarios, we compute the variance $\V_\ell=\V[I_{N_\ell}^{m,\ell}]$ (cf.~equation \eqref{eq:mlqmc_level_estimator}) of the (ML)QMC estimator on MLQMC level $\ell$ by using $M=128$ ($M=64$ in 3D) randomizations of the Sobol' sequence used and we monitor the quantity $\log_2(N\V_\ell)$ as the number of samples $N$ is increased. Various line styles are used in figure \ref{fig:QMC_conv} to indicate the different sample sizes. The horizontal lines correspond to QMC and the oblique lines to MLQMC.

For standard MC and MLMC, we have $\V_\ell = O(N^{-1})$, giving a $\log_2(N\V_\ell)$ of $O(1)$. For this reason, if we were observing a MC-like convergence rate, we would see the different lines of figure \ref{fig:QMC_conv} overlapping. The fact that this does not happen means that we are in fact observing a QMC-like rate which is faster than $O(N^{-1})$ (for the variance). However, it is clear by looking at figures \ref{fig:QMC_conv3} and \ref{fig:QMC_conv5} that as $N$ grows the lines get closer to each other, marking a decay to a $O(N^{-1})$ rate of convergence (for the variance) as predicted by theorem \ref{th:hybrid_QMC_conv} (see supplementary material). By comparing the figures on the left hand side to those on the right hand side, it is also clear that increasing the Haar level delays the occurrence of this behaviour both in the QMC case and in the MLQMC case. Furthermore, it appears that in the MLQMC case the convergence rate decays sooner than in the QMC case. Finally, we note that MLQMC indeed benefits from the combination of QMC and MLMC: the variance of the MLQMC estimator on any level is always smaller than the corresponding QMC estimator for the same number of samples, with large variance reductions on the fine levels.

In figure \ref{fig:QMC_conv_part_2} we take a closer look at one of the 2D examples from figure \ref{fig:QMC_conv} and we compare the observed (ML)QMC convergence rate with the rate that would be expected in a standard MC regime. We monitor the standard deviation of the (ML)QMC estimators for $P_\ell$ (left) and $P_\ell-P_{\ell-1}$ (right). Initially, a QMC-like convergence rate is observed for both quantities, which eventually decays to a standard MC rate. This transition occurs later when the Haar level is larger. To see this, compare the different lines (in this test case a higher (ML)QMC level corresponds to a higher Haar level) in figure \ref{fig:QMC_conv_part_2}.

We now focus on the 2D case only for simplicity and see how both QMC and MLQMC perform in practice when applied to equation \eqref{eq:diffusion_eqn_for_QMC_conv_bis}. In figure \ref{fig:MLQMC_conv_1} we study the adaptivity and cost of (ML)QMC as the root mean square error tolerance $\varepsilon$ is decreased. We choose Haar meshes with mesh size comparable to the FE mesh size ($|\square_{\mathscr{L}_\ell}|^{1/2} = 2^{-\ell}$) and we fix the number of randomizations to be $M=32$. In the plot on the left we see how MLQMC automatically selects the number of samples according to the greedy strategy highlighted in the supplementary material \ref{secSM:multilevel_methods} and in \cite{GilesWaterhouse2009} so as to satisfy the given error tolerance. As in the MLMC case, more samples are taken on coarse levels and only a few on the fine levels.

\begin{figure}[h!]
	\centering
	\begin{subfigure}{.49\textwidth}
		\includegraphics[width=\textwidth, trim={0cm 0cm 0cm 0cm},clip]{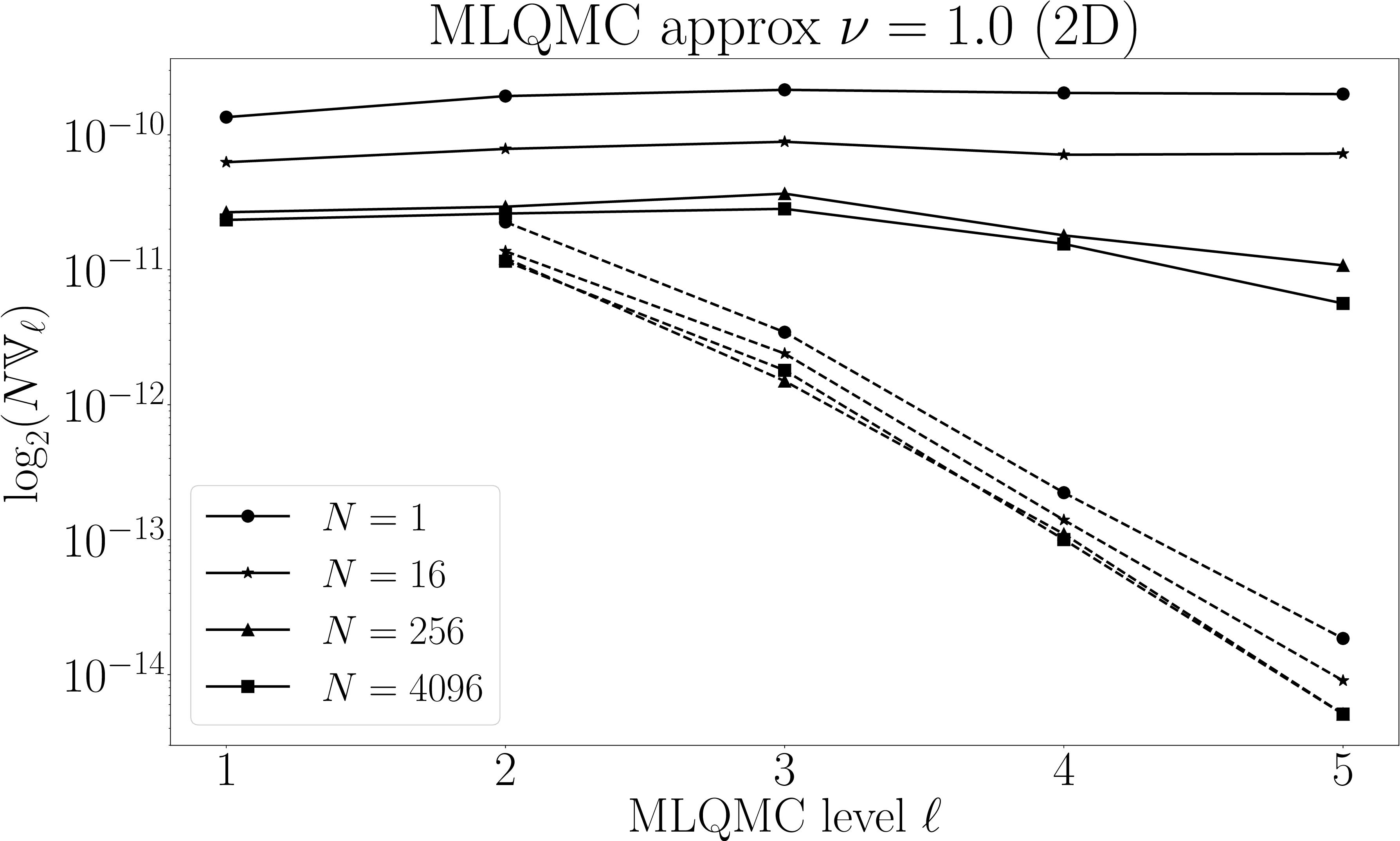}
		\caption{}
		\label{fig:QMC_conv3}
	\end{subfigure}
	\begin{subfigure}{.49\textwidth}
		\includegraphics[width=\textwidth, trim={0cm 0cm 0cm 0cm},clip]{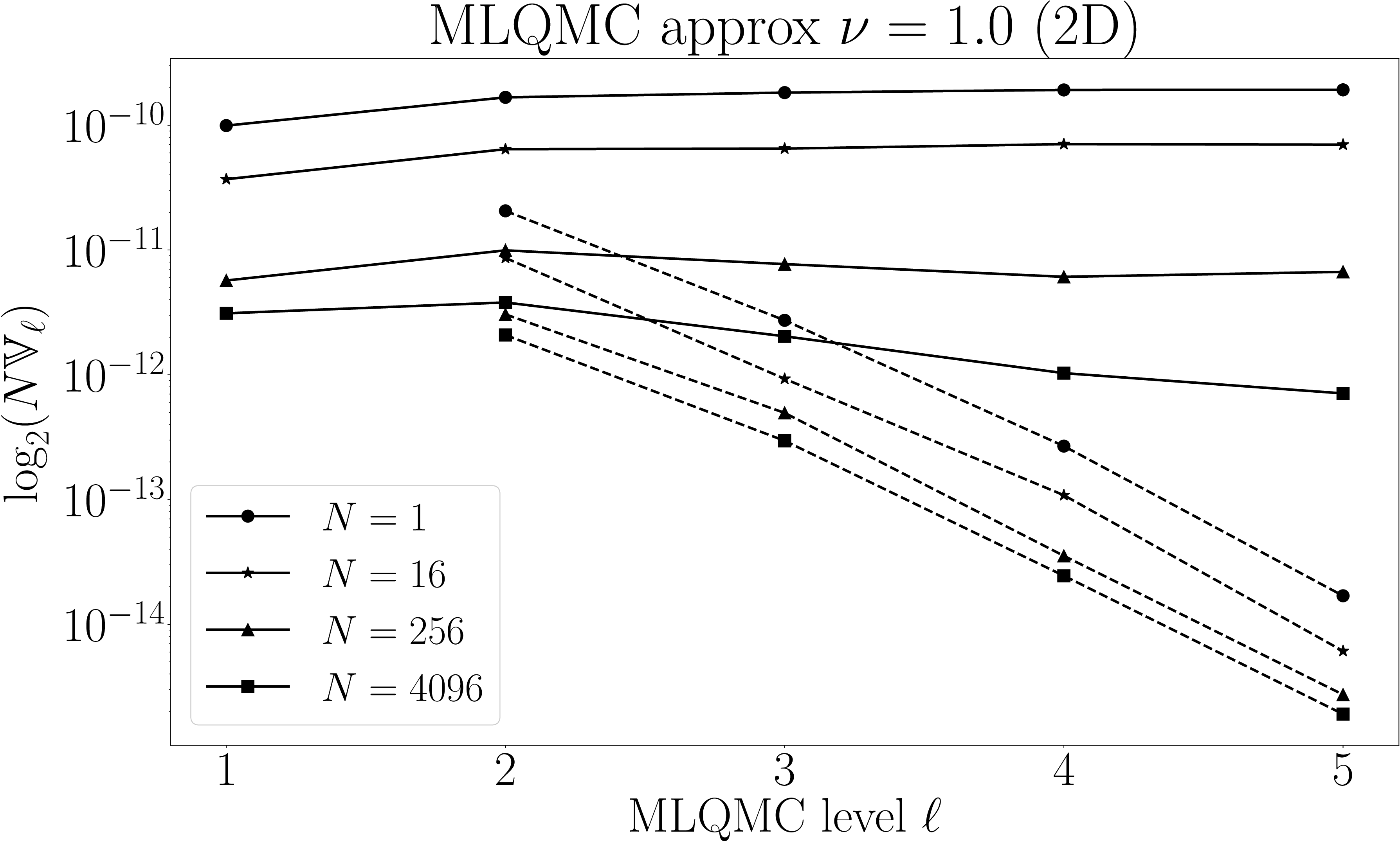}
		\caption{}
		\label{fig:QMC_conv4}
	\end{subfigure}\\
	\begin{subfigure}{.49\textwidth}
		\includegraphics[width=\textwidth, trim={0cm 0cm 0cm 0cm},clip]{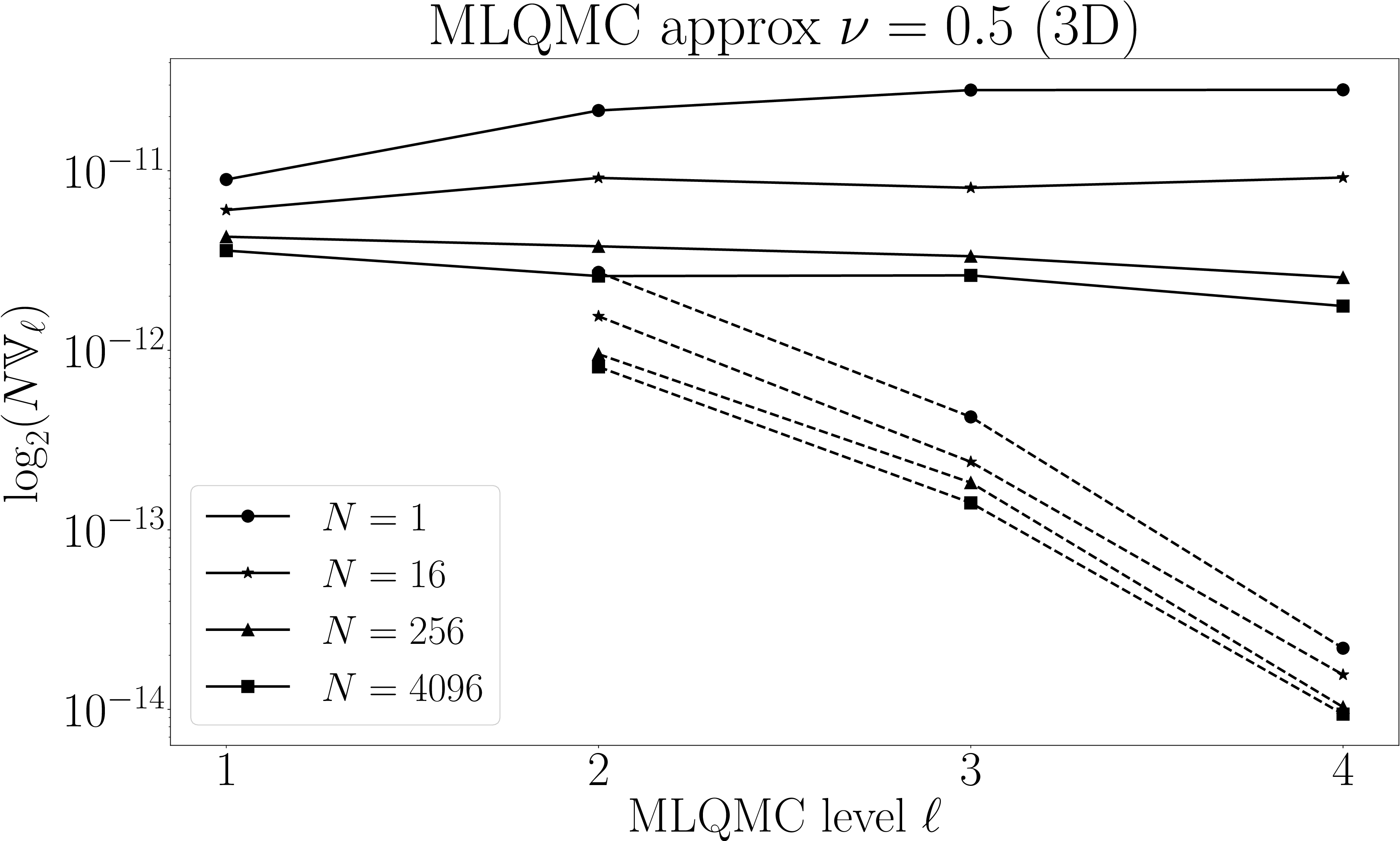}
		\caption{}
		\label{fig:QMC_conv5}
	\end{subfigure}
	\begin{subfigure}{.49\textwidth}
		\includegraphics[width=\textwidth, trim={0cm 0cm 0cm 0cm},clip]{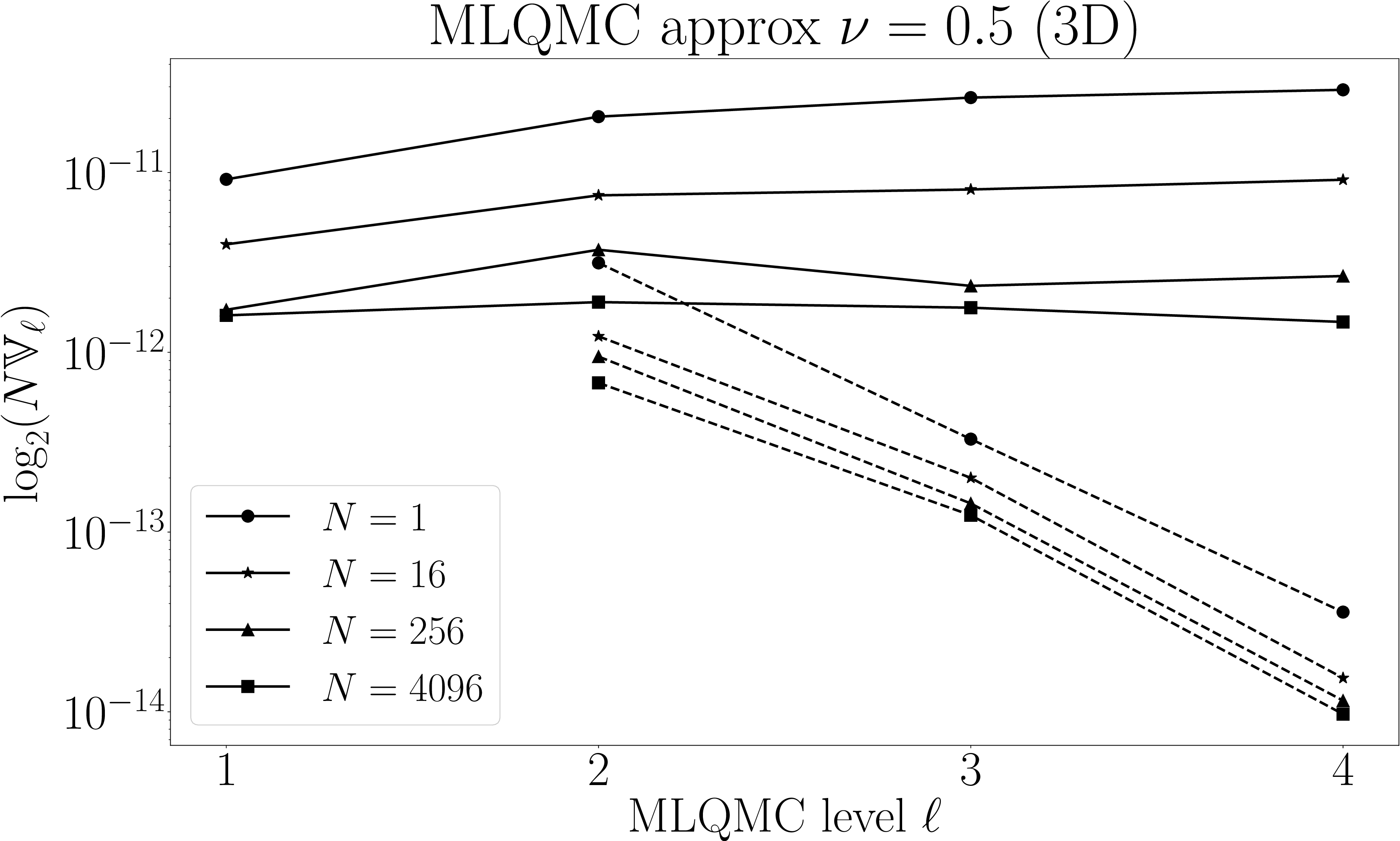}
		\caption{}
		\label{fig:QMC_conv6}
	\end{subfigure}
	\vspace{-16pt}
	\caption{\textit{Convergence behaviour of (ML)QMC with respect to the number of samples $N$ in and 2D (a)-(b) ($M=128$), and in 3D (c)-(d) ($M=64$). The plots on the left are obtained by choosing Haar mesh sizes $|\square_{\mathscr{L}_\ell}|^{1/d} =2^{-\ell}$ in 2D (a) and 3D (c). The plots on the right are obtained by choosing Haar meshes with $|\square_{\mathscr{L}_\ell}|^{1/2} =2^{-(\ell+2)}$ in 2D (b) and $|\square_{\mathscr{L}_\ell}|^{1/3} = 2^{-(\ell+1)}$ in 3D (d). The (approximately) horizontal and oblique lines correspond to QMC and MLQMC respectively. Different markers indicate different sample sizes. On the $y$-axis we monitor (the logarithm of) $N\V_\ell$. 
	}}
	\label{fig:QMC_conv}
	\vspace{-24pt}
\end{figure}

\begin{figure}[h!]
	\vspace{-8pt}
	\centering
	\vspace{6pt}
	\begin{subfigure}{.8\textwidth}
		\includegraphics[width=\textwidth, trim={0cm 0cm 0cm 0cm},clip]{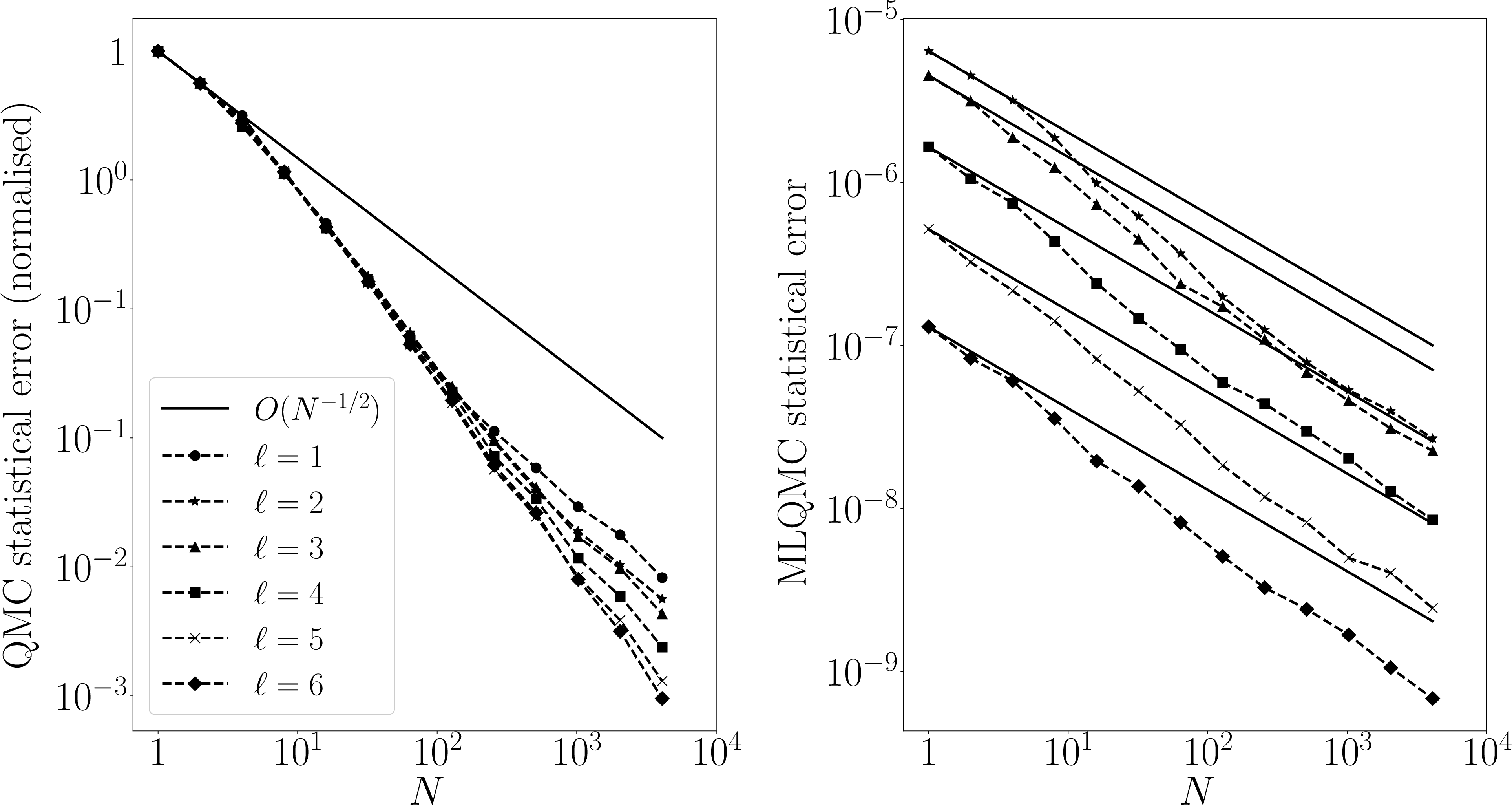}
	\end{subfigure}
	\caption{\textit{Convergence behaviour of (ML)QMC with respect to the number of samples $N$ in 2D. The plots correspond to the same problem and (ML)QMC set-up considered in Figure \ref{fig:QMC_conv4}. The continuous lines correspond to the standard MC rate. In the figure on the left, we plot the standard deviation (SD) of the QMC estimator for $P_\ell$ as a function of the number of samples. In order to make the figures clearer, we normalise this value by dividing by the SD of the estimator when $N=1$. The actual values can be recovered by looking at the horizontal lines in Figure \ref{fig:QMC_conv4}. In the figures on the right, we plot the SD of the MLQMC estimator for $P_\ell-P_{\ell-1}$ as a function of $N$ (dashed lines) and we compare them with the theoretical standard MC convergence behaviour (continuous lines). Overall, we observe a pre-asymptotic QMC rate that eventually decays to a $O(N^{-1/2})$ standard MC rate.
	}}
	\label{fig:QMC_conv_part_2}
	\vspace{-21pt}
\end{figure}

\begin{figure}[h!]
	\centering
	\begin{subfigure}{.8\textwidth}
		\includegraphics[width=\textwidth, trim={0cm 0cm 0cm 0cm},clip]{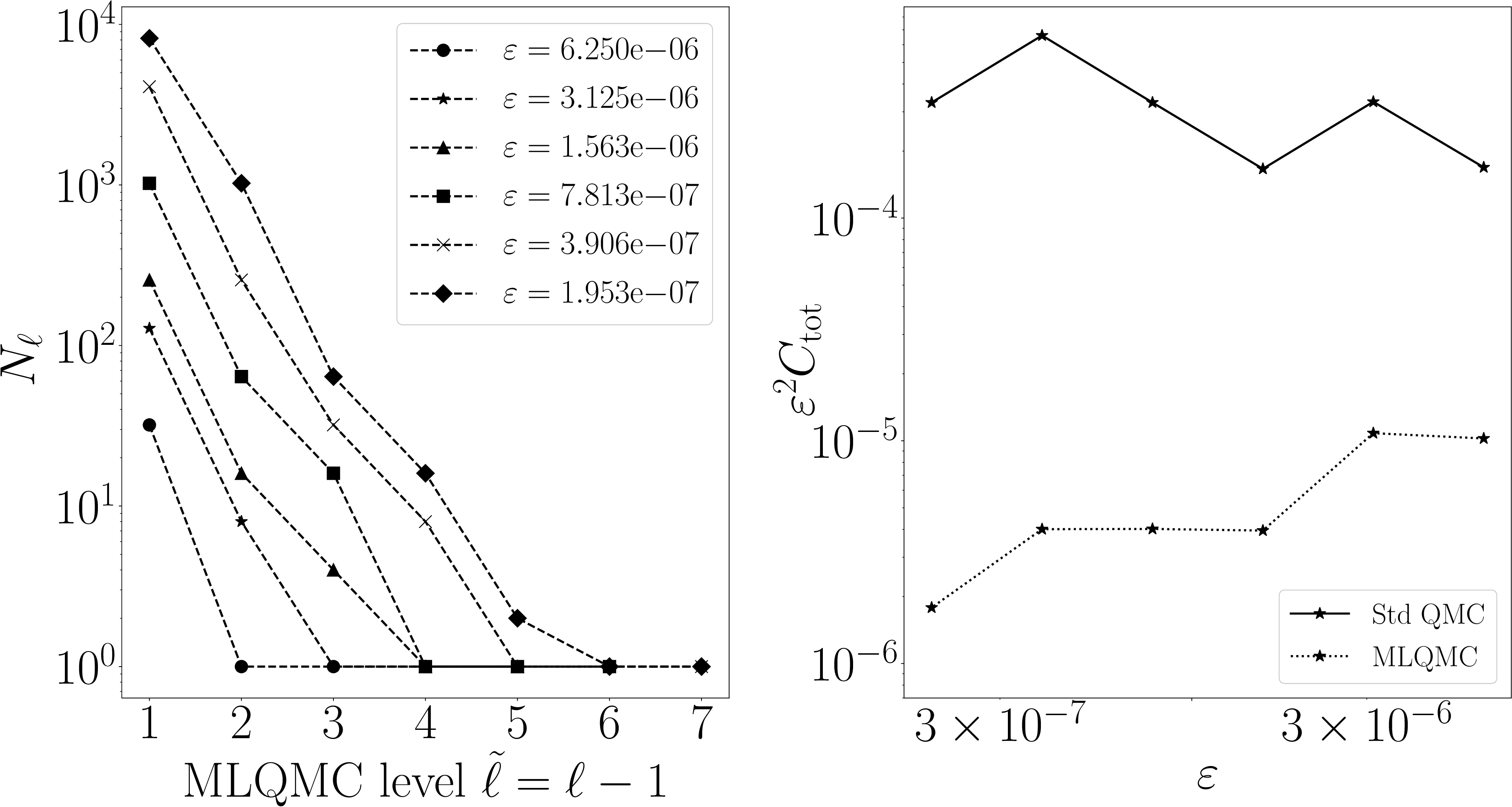}
	\end{subfigure}
	\caption{\textit{MLQMC convergence for the solution of \eqref{eq:diffusion_eqn_for_QMC_conv_bis}. Here $M=32$ and $\mathscr{L}_\ell = 2 + \ell$. Left plot: we show how the MLQMC algorithm automatically selects the optimal number of samples $N_\ell$ on each level to achieve a given tolerance $\varepsilon$. Note that we have dropped the first mesh of the hierarchy as it is too coarse and it would not bring any significant advantage to the performance of MLQMC (same reasoning as for MLMC, see section 2.6 in \cite{giles2015multilevel}). We observe that on the finest levels only one sample is used, making MLQMC equivalent to plain MLMC on these levels. Right plot: we compare the efficiency of MLQMC with QMC for different tolerances. MLQMC appears to have a better-than-$O(\varepsilon^{-2})$ total cost complexity and significantly outperforms QMC.}}
	\label{fig:MLQMC_conv_1}
	\vspace{-8pt}
\end{figure}

\begin{figure}[h!]
	\centering
	\includegraphics[width=0.86\textwidth, trim={0cm 0cm 0cm 0cm},clip]{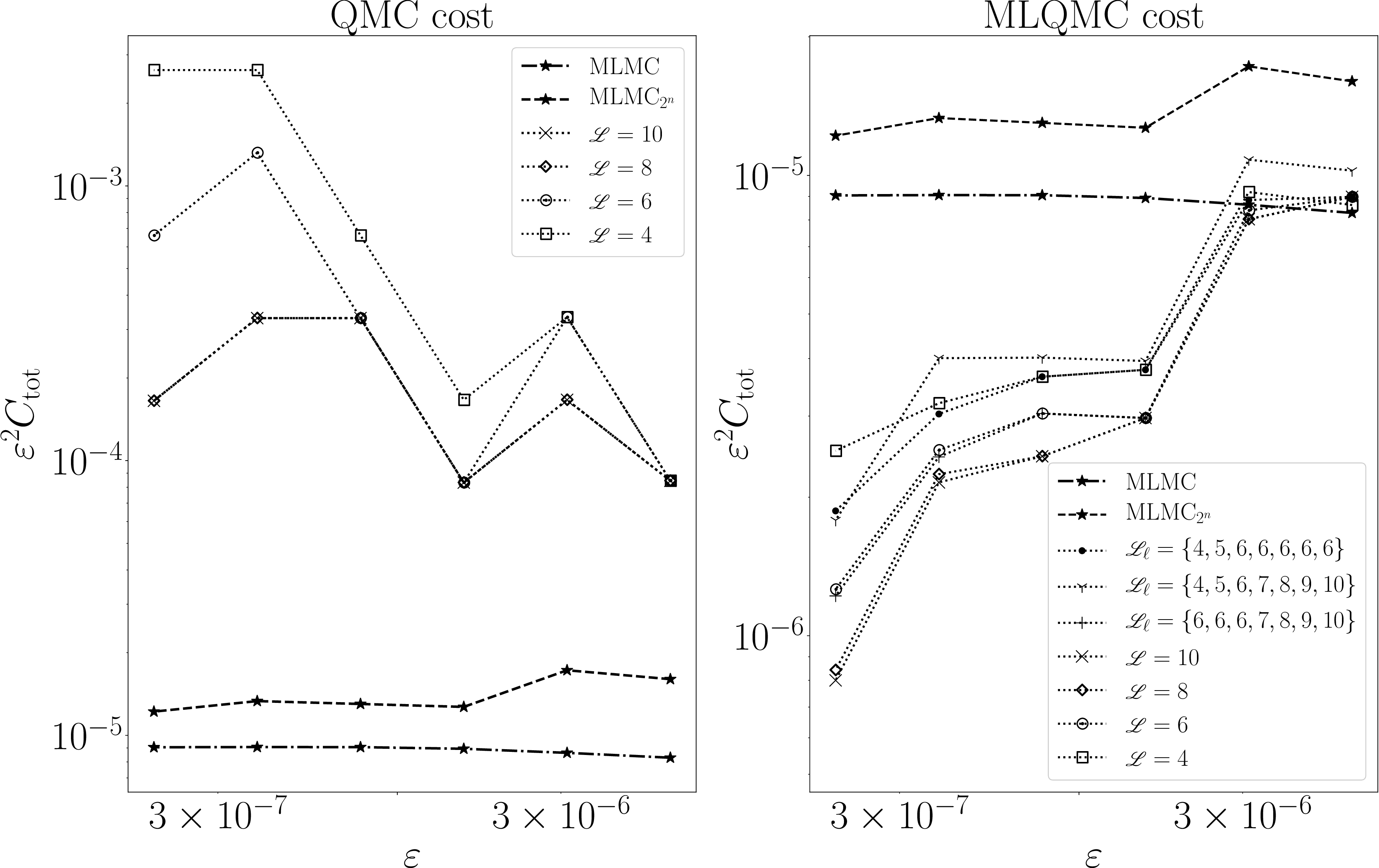}
	\caption{\textit{MLMC, QMC and MLQMC total computational cost needed for the solution of \eqref{eq:diffusion_eqn_for_QMC_conv_bis} with the same FE mesh hierarchy as in figure \ref{fig:MLQMC_conv_1}. In the (ML)QMC case, we take $M=32$ and consider different Haar level hierarchies which correspond to different computational costs. The $x$-axis and the MLMC lines are the same in both plots. MLQMC outperforms MLMC which in turn outperforms QMC.}}
	\label{fig:MLQMC_conv_2}
	\vspace{-20pt}
\end{figure}

In the plot on the right we show the overall cost of QMC and MLQMC as the root mean square error tolerance $\varepsilon$ is reduced. The total cost is computed by taking the cost on each level to be $C_\ell=c_3 2^{\gamma\ell}$, where $c_3$ and $\gamma$ are estimated using CPU timings. More specifically, we plot the quantity $\varepsilon^2C_{\text{tot}}$, where $C_{\text{tot}}$ is the total cost. The reason is that the total cost complexity of MLMC for this problem (MLMC parameters: $\beta = 2\alpha = 4$, $\gamma=2$, cf.~supplementary material \ref{secSM:multilevel_methods}) is $O(\varepsilon^{-2})$, giving the $\varepsilon^2C_{\text{tot}}$ factor to be $O(1)$ for all $\varepsilon$. The fact that the MLQMC cost line is not horizontal, but decreases as $\varepsilon$ is reduced shows that the total complexity of MLQMC is better than $\varepsilon^{-2}$, i.e.~that our MLQMC algorithm has a better-than-MLMC complexity. This improved complexity stems from the fact that we are observing a QMC-like convergence rate with respect to $N_\ell$.

As $\varepsilon$ is decreased, we expect the cost complexity to decay to an $\varepsilon^{-2}$ rate: for extremely fine tolerances very large sample sizes are required yielding the asymptotic $O(N^{-1/2})$ standard MC rate and harming the overall cost complexity. However, even in this case, the overall MLQMC cost benefits from the pre-asymptotic regime and MLQMC still outperforms MLMC (see figure \ref{fig:MLQMC_conv_2}). Similarly, QMC initially benefits from a faster convergence rate with respect to $N$. As the tolerance is decreased, the QMC rate decays to a standard MC rate and the total cost of QMC starts increasing faster than $O(\varepsilon^{-2})$.

We investigate the performance of (ML)QMC as the Haar levels are varied in figure \ref{fig:MLQMC_conv_2}, where we show the total cost of (ML)QMC for different Haar level hierarchies and we compare it with the cost of standard MLMC. The $x$-axis and the black lines in both plots are the same. We present the costs of two versions of MLMC: the black dash-dotted line corresponds to standard MLMC, while the black dashed line corresponds to a MLMC algorithm in which the number of samples are restricted to be in powers of $2$ (this restriction also applies to the MLQMC algorithm we use \cite{GilesWaterhouse2009}). MLQMC outperforms MLMC by a factor of approximately $8$, depending on the Haar level choice.

Recall the convergence results with respect to the number of samples shown in figure \ref{fig:QMC_conv}: even if the convergence rate decreases as $N$ increases, it is clear from figure \ref{fig:MLQMC_conv_1} (left) that this only happens on the coarse levels where more samples are needed. Since for problem \eqref{eq:diffusion_eqn_for_QMC_conv_bis} and the FEM discretization chosen we are in the ``good'' case of the MLMC theorem (i.e.~$\beta>\gamma$, cf.~theorem \ref{th:MLMC} in the supplementary material), the multilevel cost is dominated by the sample cost on the coarse levels. We therefore expect to obtain computational gains by increasing the Haar level on the coarse MLQMC levels. At the same time, we do not expect to lose in computational efficiency if we decrease the Haar level on the fine levels as these are not dominating the total cost. Note that in the QMC case there is only one level and the only possible strategy is to keep the Haar level as high as required.

\begin{remark}
	\label{rem:MLQMC_cost_no_supermesh_cells}
	The results shown in figures \ref{fig:MLQMC_conv_1} and \ref{fig:MLQMC_conv_2} do not account for differences in the cost per sample due to variations in the number of supermesh cells. If the cost of solving the PDE with random coefficients of interest (e.g.~equation \eqref{eq:diffusion_eqn_for_QMC_conv_bis}) dominates over the cost of sampling white noise realizations, these results are still valid as is. Otherwise, extra care must be taken when using Haar meshes which are much finer than the corresponding FE meshes as this results in a large number of supermesh cells. In the figures this would apply to Haar levels greater than $\mathscr{L}_\ell = \{4,\dots,10\}$ (gray line in the plot on the right) and there is clearly a trade-off: larger $\mathscr{L}$ means faster decay with respect to $N$, but larger costs per sample as well.
\end{remark}

By looking at figure \ref{fig:MLQMC_conv_2} it is clear that our expectations are met. In the QMC case (plot on the left) we see that a small Haar level results in significant cost increase for small tolerances, while for large Haar levels we retain good convergence with respect to $N$ and a cost complexity which looks just slightly worse than $O(\varepsilon^{-2})$. For this specific problem, the optimal strategy would be to increase the Haar level as the mesh is refined so that the Haar level is increased only when needed. Generally speaking, we believe that it is never advantageous to use Haar meshes much finer than the FE meshes (cf.~remark \ref{rem:MLQMC_cost_no_supermesh_cells}).

In the MLQMC case (plot on the right in figure \ref{fig:MLQMC_conv_2}), we note that increasing the Haar level on the coarse levels indeed brings computational advantages and decreasing it on the fine levels does not seem to affect the total cost, as predicted. For this specific problem the optimal strategy therefore consists of increasing the Haar level on the coarse levels and either keeping it constant across the MLQMC hierarchy or possibly even decreasing it. For the 2D MLQMC hierarchy, a good choice seems to fix $\mathscr{L}_\ell = 6$ for all $\ell$ since a larger Haar level would significantly increase the number of supermesh cells (cf.~remark \ref{rem:MLQMC_cost_no_supermesh_cells}). A brief discussion about how the Haar levels should be chosen for more generic problems is given in remark \ref{rem:generic_problems_Haar_level_choice}.

Overall, our MLQMC strategy outperforms MLMC, which in turn outperforms QMC. Standard MC is always worse than QMC, by up to two orders of magnitude for small $\varepsilon$ (not shown).

\begin{remark}
	\label{rem:generic_problems_Haar_level_choice}
	The optimal Haar strategy is likely to change if the problem to be solved belongs to the other two cases of the MLMC theorem (theorem \ref{th:MLMC} in the supplementary material), i.e.~$\beta=\gamma$ or $\beta<\gamma$. In the first case  ($\beta = \gamma$), the total multilevel cost is simultaneously dominated by all levels in the multilevel hierarchy and we expect in this case that the optimal strategy is to use a Haar mesh hierarchy of mesh sizes comparable to the FE mesh sizes (e.g.~as for the gray line in the right plot of figure \ref{fig:MLQMC_conv_2}). In the latter case ($\beta < \gamma$), the total multilevel cost is dominated by the fine levels. In this case it might be advantageous to keep the Haar level low on the coarse levels and to increase it on the fine levels.
\end{remark}

\section{Conclusions}
\label{sec:MLQMC_conclusions}
We presented a novel algorithm to efficiently compute the action of white noise and sample Mat\'ern fields within a QMC and MLQMC framework. This algorithm retains the computational efficiency of the MLMC case \cite{Croci2018} and still enforces the required multilevel coupling in a non-nested mesh hierarchy. The numerical results show that our technique works well in practice, that the FEM convergence orders observed agree with the theory and that MLQMC outperforms MLMC and has a better cost complexity in the pre-asymptotic regime. We remark that while the numerical examples considered here only employ (non-nested) structured grids, our QMC algorithm has also successfully been applied to unstructured MRI-derived brain meshes in \cite{CrociVinjeRognes2019bis}. We note that the sampling technique presented extends naturally to
any application in which spatial white noise realizations are needed within a finite element
framework provided that the solution is smooth enough.
An open problem is the derivation of a closed-form expression for the optimal number of samples on each MLQMC level and for the optimal Haar level hierarchy, but we leave this to future research. It would also be interesting to extend the algorithms presented to general higher degree wavelets and domains (cf.~remarks \ref{rem:QMC_generalizations_wavelets} and \ref{rem:QMC_generalizations}). The first enhancement (generic wavelets) could possibly improve the convergence rate with respect to the number of samples, while the second (general domains) would reduce the supermeshing complexity and consequently the white noise sample cost.

\section*{Acknowledgments}
The authors would like to acknowledge useful discussions with Marie Rognes, Robert Scheichl, Endre S\"uli, Abdul-Lateef Haji-Ali, Alberto Paganini and Casper Beentjes. The authors would also like to express their thanks to James R. Maddison for his assistance with the implementation of an interface between libsupermesh and FEniCS.

\printbibliography

\newpage

\appendix
\setcounter{page}{1}

\renewcommand{\appendixname}{}

\renewcommand{\appendixpagename}{\centering \vspace{12pt}\LARGE{Supplementary material}\\\vspace{6pt}\Large{Multilevel quasi Monte Carlo methods for elliptic PDEs with random field coefficients via fast white noise sampling}\\\vspace{6pt}\small{M.~Croci, M.~B.~Giles, P.~E.~Farrell}}
\appendixpage

\section{Multilevel Monte Carlo methods}
\label{secSM:multilevel_methods}

\subsection{MLMC convergence theory}

We briefly summarize the MLMC convergence theory since it is useful to understand some of the considerations drawn in section \ref{sec:MLQMC_num_res}.

The MSE of the MLMC estimator is given by,
\begin{align}
MSE = \hat{V} + E[\hat{P} - P]^2,
\end{align}
where $\hat{P}$ is the MLMC estimator of variance $\hat{V}$. To ensure that $MSE\leq\varepsilon^2$, we enforce the bounds,
\begin{align}
\label{eq:theta}
\hat{V}\leq (1-\theta) \varepsilon^2,\quad E[\hat{P} - P]^2 \leq \theta \varepsilon^2,
\end{align}
where $\theta\in (0,1)$ is a weight, introduced by Haji-Ali et al.~in \cite{haji2016optimization}. Let $C_\ell$, $V_\ell$ be the cost and variance of one sample $(P_\ell - P_{\ell-1})(\omega)$ respectively. Then the total MLMC cost and variance are
\begin{align}
C_{tot} = \sum\limits_{\ell = 1}^LN_\ell C_\ell,\quad \hat{V} = \sum\limits_{\ell = 1}^LN_\ell^{-1}V_\ell.
\end{align}
We can minimise the estimator variance for fixed total cost. For further details we refer to \cite{giles2015multilevel}. This gives that, for a fixed MSE tolerance $\varepsilon^2$, the optimal number of samples for each level and related total cost are,
\begin{align}
\label{eq:optimal_number_of_samples}
N_\ell = \left(\varepsilon^{-2}\sum\limits_{\ell = 1}^L\sqrt{V_\ell C_\ell}\right)\sqrt{V_\ell/C_\ell},\quad
C_{tot} = \varepsilon^{-2}\left(\sum\limits_{\ell = 1}^L\sqrt{V_\ell C_\ell}\right)^2.
\end{align}

We can now compare the cost complexity of standard and multilevel Monte Carlo for the estimation of $\E[P_L]$. Usually, $\V[P_L]=O(V_1)$, then the total cost complexity of standard MC is $O(\varepsilon^{-2}V_1C_L)$. According to how the product $V_\ell C_\ell$ varies with level, we can have three different scenarios for MLMC:
\begin{enumerate}[leftmargin=1cm]
	\item The product $V_\ell C_\ell$ increases with level ($\gamma > \beta$ in theorem \ref{th:MLMC}, to follow). Then, to leading order, the total MLMC cost is $O(\varepsilon^{-2}V_LC_L)$, for an improvement in computational cost over standard Monte Carlo by a $V_1/V_L$ factor. In this case the total cost is dominated by the fine levels.
	\item The product is asymptotically constant with the level. Then, we have a MLMC total cost of order $O(\varepsilon^{-2}L^2V_LC_L) = O(\varepsilon^{-2}L^2V_1C_1)$ ($\gamma = \beta$ in theorem \ref{th:MLMC}, to follow). This gives a cost improvement of $V_1/(L^2V_L)\approx C_L/(L^2C_1)$ with respect to standard MC. The MLMC cost here is equally dominated by all the levels in the hierarchy.
	\item The product decreases with the level ($\gamma < \beta$ in theorem \ref{th:MLMC}, to follow). Then, $C_{tot} \approx O(\varepsilon^{-2}V_LC_1)$, for an improvement of $C_L/C_1$. For example this could be the ratio between a fine mesh PDE solution cost and a coarse mesh PDE solution cost, which is generally quite significant. In this case, the MLMC cost is dominated by the coarser levels.
\end{enumerate}

The convergence of MLMC is ensured by the following theorem.
\begin{theorem}[\cite{giles2015multilevel}, theorem $1$]
	\label{th:MLMC}
	Let $P\in L^2(\Omega,\R)$ and let $P_\ell$ be its level $\ell$ approximation. Let $Y_\ell$ be the MC estimator of $\E[P_\ell - P_{\ell-1}]$ on level $\ell$ such that
	\begin{align}
	\E[Y_\ell] = \E[P_\ell - P_{\ell-1}],
	\end{align}
	with $P_{0} = 0$, and let $C_\ell$ and $V_\ell$ the expected cost and variance of each of the $N_\ell$ Monte Carlo samples needed to compute $Y_\ell$. If the estimators $Y_\ell$ are independent and there exist positive constants $\alpha$, $\beta$, $\gamma$, $c_1$, $c_2$, $c_3$, such that $\alpha\geq \frac{1}{2}\min(\beta,\gamma)$ and
	\begin{align}
	\label{eq:MLMC_theorem_bounds}
	|\E[P_\ell - P]|\leq c_12^{-\alpha\ell},\qquad V_\ell\leq c_2 2^{-\beta\ell},\qquad C_\ell\leq c_32^{\gamma\ell},
	\end{align}
	then there exist a positive constant $c_4$ such that, for all $\varepsilon < e^{-1}$, there is a level number $L$ and number of samples $N_\ell$, such that the MLMC estimator
	\begin{align}
	\hat{P}=\sum\limits_{\ell=1}^L Y_\ell,
	\end{align}
	has MSE with bound,
	\begin{align}
	MSE = \E[(\hat{P} - \E[P])^2]\leq \varepsilon^2,
	\end{align}
	with a total computational complexity with bound,
	\begin{align}
	\E[C_{tot}] \leq \left\{\begin{array}{lr} c_4\varepsilon^{-2}, & \beta > \gamma,\\c_4\varepsilon^{-2}(\log\varepsilon)^2, & \beta = \gamma,\\c_4\varepsilon^{-2 - (\gamma - \beta)/\alpha}, & \beta < \gamma.\end{array}\right.
	\end{align}
\end{theorem}

\subsection{MLQMC algorithm}

Let $C_\ell$ be the cost of evaluating $Y_\ell$ and let $\V_\ell=\V[I_{N_\ell}^{m,\ell}]$, where $I_{N_\ell}^{m,\ell}$ is given in \eqref{eq:mlqmc_level_estimator}. The MLQMC algorithm we adopt, taken from \cite{GilesWaterhouse2009}, is the following.\vspace{6pt}

\paragraph{MLQMC algorithm (taken from \cite{GilesWaterhouse2009})}
\begin{enumerate}[leftmargin=1.5cm]
	\item Set the required tolerance $\varepsilon$, $\theta\in(0,1)$, the minimum and maximum level $L_{\min}$ and $L_{\max}$ and the initial number of levels to be $L=1$.
	\item Get an initial estimate of $\V_L$ with $N_L=1$ and $M=32$ randomizations.
	\item While $\sum\limits_{\ell=1}^L \V_\ell > (1-\theta)\varepsilon^2$, double $N_\ell$ on the level with largest $\V_\ell/(C_\ell N_\ell)$.
	\item If $L<L_{\min}$ or the bias estimate is greater than $\sqrt{\theta}\varepsilon$, set $L=L+1$. If $L \leq L_{\max}$ go to 2, otherwise report convergence failure.
\end{enumerate}

\begin{remark}[Adapted from \cite{GilesWaterhouse2009}]
	The term $\sum_{\ell=1}^L \V_\ell$ is the total estimator variance and the variable $\theta$ is a weight with the same meaning as in the MLMC case. The choice of $N_\ell$ on each level is heuristic: doubling the number of samples will eliminate (independently on whether we are in an MC or QMC-like convergence rate regime) most of the estimated variance $\V_\ell$ on level $\ell$ at a cost $N_\ell C_\ell$ and we therefore double the number of samples on the level that offers the largest variance reduction per unit cost.
\end{remark}

\section{A partial QMC convergence result}
\label{sec:partial_QMC_convergence_result}

The white noise sampling strategy we presented is hybrid in the sense that both randomized QMC and pseudo-random sequences are used. It is then unclear what the order of convergence with respect to the number of samples should be. The hope is, of course, to achieve something better than the standard MC rate of convergence. Proving convergence results with respect to the number of samples for QMC (and MLQMC) in a PDE setting is non-trivial and results have so far been established only for a limited class of QMC point sequences \cite{KuoScheichlSchwabEtAl2017,HerrmannSchwab2017}.

Although deriving a convergence estimate is outside of the scope of this work, in what follows we try to build up intuition about what is likely to be happening in practice. As an example problem, consider the following elliptic PDE with log-normal diffusion coefficient (already in weak form): find $p(\cdot,\omega)\in H^1_0(G)$ such that for all $v\in H^1_0(G)$,
\begin{align}
\label{eq:diffusion_eqn_for_QMC_conv}
a(p,v) = (D^*(\bm{x},\omega)\nabla p, \nabla v) = (f,v),\quad \text{a.s.},\quad D^*(\bm{x},\omega) = e^{u(\bm{x},\omega)},
\end{align}
where for a given sampling domain $D\subset\joinrel\subset \R^d$ we have that $G\subset\joinrel\subset D$ is a domain of class $C^{1,\epsilon}$ for any $\epsilon>0$, $f\in L^\infty(G)$ (these hypotheses imply $p\in W^{1,\infty}(G)$ a.s., see theorem $8.34$ in \cite{GilbargTrudinger}) and $u(\bm{x},\omega)\in H^{\nu-\epsilon}(G) \cap C^{0}(\bar{G})$ \cite{Bolin2017,Buckdahn1989} is a Mat\'ern field satisfying equation \eqref{eq:white_noise_SPDE_reminder} over $D$. In addition, assume that we are interested in computing the expectation of a possibly nonlinear output functional of $p$, namely $\mathcal{P}(p)$. We will now establish the following result:
\begin{theorem}
	\label{th:hybrid_QMC_conv}
	Let $p\in W^{1,\infty}(G)$ a.s.~be the solution of \eqref{eq:diffusion_eqn_for_QMC_conv} where $u$ has been sampled using the hybrid QMC technique presented in this paper. Let $u_\mathscr{L}$ be the solution of \eqref{eq:white_noise_SPDE_reminder} obtained by using the same $\W_\mathscr{L}$ sample as for $u$ and by setting $\W_R=0$, and let ${p_\mathscr{L}(\cdot,\omega)\in H^1_0(G)}$ satisfy, for all $v\in H^1_0(G)$,
	\begin{align}
	a_\mathscr{L}(p_\mathscr{L},v) = (D^*_\mathscr{L}(\bm{x},\omega)\nabla p_\mathscr{L}, \nabla v) = (f,v),\ \ \text{a.s.},\quad D^*_\mathscr{L}(\bm{x},\omega) = e^{u_\mathscr{L}(\bm{x},\omega)}.
	\end{align}
	Let $\mu=\min(\nu,1)$ and let $\mathscr{L}$ be the Haar level used to sample $u$. Assume that there exist $\alpha,\beta\geq 1$ independent from $\mathscr{L}$ such that $||u_{\mathscr{L}}||_{L^{\infty}(G)}\leq \alpha||u||_{L^{\infty}(G)}$ and $||\nabla p_{\mathscr{L}}||_{L^{\infty}(G)}\leq \beta||\nabla p||_{L^{\infty}(G)}$. Further assume that the functional $\mathcal{P}$ is continuously Fr\'echet differentiable and let $\hat{P}, \hat{P}_\mathscr{L}$ and $\widehat{P - P_\mathscr{L}}$ be the QMC estimators for $\E[P]$, $\E[P_\mathscr{L}]$ and $\E[P - P_\mathscr{L}]$ respectively, obtained by using $N$ QMC points. Here, $P = \mathcal{P}(p)$ and $P_\mathscr{L} = \mathcal{P}(p_\mathscr{L})$. If there exist constants $c>0$, $q\geq 1$, and $N_0\geq1$ such that the QMC estimators satisfy for $N>N_0$,
	\begin{align}
	\label{eq:th_hybrid_QMC_hp}
	\V[\hat{P}]\leq c\frac{\V[P]}{N},\quad \V[\hat{P}_\mathscr{L}]\leq c\frac{\V[P_\mathscr{L}]}{N^q},\quad \V[\widehat{P-P_\mathscr{L}}]\leq c\frac{\V[P-P_\mathscr{L}]}{N},
	\end{align}
	i.e.~the QMC estimators are never asymptotically worse than standard Monte Carlo, then there exists a sufficiently large integer constant $\mathscr{L}_0$ such that the statistical error $\V[\hat{P}]$ also satisfies for $N> N_0$, $\mathscr{L}>\mathscr{L}_0$,
	\begin{align}
	\V[\hat{P}] \leq \hat{c}\frac{\V[P]}{N^q} + \frac{\bar{c}}{N}2^{-\mu \mathscr{L}},
	\end{align}
	where $\hat{c},\bar{c} > 0$ are independent from $\mathscr{L}$ and $N$.
\end{theorem}

\begin{remark}
	Condition \eqref{eq:th_hybrid_QMC_hp} is satisfied by most randomized QMC point sets, e.g.~randomized digital nets and sequences and randomly shifted lattice rules \cite{Owen1995,LEcuyer2000}.
\end{remark}

\begin{proof}
	In this proof, we work pathwise for fixed $\omega\in\Omega$. We start with essentially the same duality argument that yields lemma $3.2$ in \cite{TeckentrupMLMC2013}. Let $v$, $\bar{v} \in H^1(G)$ and let $D_v\mathcal{P}(\bar{v})$ be the Gateaux derivative of $\mathcal{P}$ at $\bar{v}$, namely
	\begin{align}
	D_v\mathcal{P}(\bar{v}) := \lim\limits_{\epsilon\rightarrow 0} \frac{\mathcal{P}(\bar{v} + \epsilon v) - \mathcal{P}(\bar{v})}{\epsilon}.
	\end{align}
	Define the average derivative of $\mathcal{P}$ on the path from $p$ to $p_\mathscr{L}$,
	\begin{align}
	\overline{D_v\mathcal{P}}(p,p_\mathscr{L}) = \int\limits_0^1D_v\mathcal{P}(p + \theta(p_\mathscr{L}-p))\text{ d}\theta,
	\end{align}
	and introduce the dual problem: find $z(\cdot, \omega)\in H^1_0(G)$ s.t.
	\begin{align}
	a(v,z) = \overline{D_v\mathcal{P}}(p,p_\mathscr{L}),\quad\forall v\in H^1_0(G).
	\end{align}
	The fundamental theorem of calculus for Fr\'echet derivatives then yields,
	\begin{align}
	P - P_\mathscr{L} = \int\limits_0^1D_{p-p_\mathscr{L}}\mathcal{P}(p + \theta(p_\mathscr{L}-p))\text{ d}\theta = \overline{D_{p-p_\mathscr{L}}\mathcal{P}}(p,p_\mathscr{L}) = a(p-p_\mathscr{L},z),\ \ \text{a.s.}
	\end{align}
	Applying H\"older's inequality gives,
	\begin{align}
	\label{eq:hybrid_qmc_proof_part1}
	|P-P_\mathscr{L}| = |a(p-p_\mathscr{L},z)| \leq ||D^*||_{L^\infty(G)}|z|_{H^1(G)}|p-p_\mathscr{L}|_{H^1(G)}\quad\text{a.s.}
	\end{align}
	We now need a bound for $|p-p_\mathscr{L}|_{H^1(G)}$. Note that, a.s. for all $v\in H^1_0(G)$,
	\begin{align}
	0 = a(p,v)-a_\mathscr{L}(p_\mathscr{L},v) = a(p-p_\mathscr{L},v) + a(p_\mathscr{L},v)-a_\mathscr{L}(p_\mathscr{L},v).
	\end{align}
	Setting $v=p-p_\mathscr{L}$ gives
	\begin{align}
	0 \leq a(p-p_\mathscr{L},p-p_\mathscr{L}) = 
	a_\mathscr{L}(p_\mathscr{L},v) - a(p_\mathscr{L},v) = ((D^*_\mathscr{L}-D^*)\nabla p_\mathscr{L},\nabla(p-p_\mathscr{L}))\quad\text{a.s.}
	\end{align}
	Both quantities can be bounded as follows: let $D^*_{\min}(\omega) = \inf_{x\in D}|D^*(\cdot,\omega)|$,
	\begin{align}
	0 \leq D^*_{\min}|p-p_\mathscr{L}|^2_{H^1_0(G)} &\leq a(p-p_\mathscr{L},p-p_\mathscr{L})\\
	((D^*_\mathscr{L}-D^*)\nabla p_\mathscr{L},\nabla(p-p_\mathscr{L})) &\leq ||D^*-D^*_\mathscr{L}||_{L^2(G)}||\nabla p_\mathscr{L}||_{L^\infty(G)}|p-p_\mathscr{L}|_{H^1(G)},
	\end{align}
	almost surely, hence yielding, after division by $|p-p_\mathscr{L}|_{H^1(G)}$,
	\begin{align}
	\label{eq:hybrid_qmc_proof_part2}
	|p-p_\mathscr{L}|_{H^1(G)} \leq \beta\frac{||D^*||_{L^\infty(G)}^\alpha}{D^*_{\min}}||\nabla p||_{L^\infty(G)}||u-u_\mathscr{L}||_{L^2(G)},\quad\text{a.s.},
	\end{align}
	for some $\alpha,\beta \geq 1$. Here we used the fact that
	by assumption
	there exists a constant $\alpha\geq 1$ independent from $\mathscr{L}$ such that ${||u_\mathscr{L}||_{L^\infty(G)} \leq \alpha||u||_{L^\infty(G)}}$, hence
	\begin{align}
	||D^*-D^*_\mathscr{L}||_{L^2(G)} &\leq e^{\max(||u||_{L^\infty(G)},||u_\mathscr{L}||_{L^\infty(G)})} ||u-u_\mathscr{L}||_{L^2(G)}\\
	&\leq ||D^*||_{L^\infty(G)}^\alpha||u-u_\mathscr{L}||_{L^2(G)},\ \text{a.s.}
	\end{align}
	Similarly, we also used the assumption that $||\nabla p_{\mathscr{L}}||_{L^\infty(G)}\leq \beta||\nabla p||_{L^\infty(G)}$.
	Putting \eqref{eq:hybrid_qmc_proof_part1} and \eqref{eq:hybrid_qmc_proof_part2} together yields,
	\begin{align}
	\label{eq:th_hybrid_QMC1}
	|P-P_\mathscr{L}| \leq C(\omega)||u-u_\mathscr{L}||_{L^2(G)},\ \ \text{a.s.},
	\end{align}
	where $C(\omega)$ is given by
	\begin{align}
	C(\omega) = \beta\frac{||D^*||_{L^\infty(G)}^{\alpha+1}}{D^*_{\min}}||\nabla p||_{L^\infty(G)}|z|_{H^1(G)}.
	\end{align}
	Note that since all terms involved are in $L^t(\Omega,\R)$ for all $t\in(0,\infty)$ (see proposition 2.4 and theorem 3.4 in \cite{CharrierMLMC2013}), we have that $C(\omega)\in L^t(\Omega,\R)$ for all $t\in[1,\infty)$ due to the generalised H\"older inequality.
	We now note that
	\begin{align}
	\label{eq:th_hybrid_QMC_step1}
	\V[\hat{P}] &= \V[\hat{P}_\mathscr{L} + \hat{P}-\hat{P}_\mathscr{L}] = \V[\hat{P}_\mathscr{L}] + \text{Cov}(\hat{P} + \hat{P}_\mathscr{L}, \hat{P} - \hat{P}_\mathscr{L}) \notag \\
	&\leq \V[\hat{P}_\mathscr{L}] + \V[\hat{P} + \hat{P}_\mathscr{L}]^{1/2}\V[(\widehat{P -P_\mathscr{L}})]^{1/2}\\
	&\leq cc_1\frac{\V[P]}{N^q} + \frac{c\tilde{c}}{N}\V[P]^{1/2}\V[P-P_\mathscr{L}]^{1/2}.
	\end{align}
	where we used the Cauchy-Schwarz inequality, hypothesis \eqref{eq:th_hybrid_QMC_hp} and the fact that since $P$ converges to $P_\mathscr{L}$ as $\mathscr{L}\rightarrow\infty$ a.s. and in $L^t(\Omega, \R)$ for $t\in[1,\infty)$, there exists a Haar level $\mathscr{L}_0$ s.t.~for all $\mathscr{L}>\mathscr{L}_0$ there exists $c_1>0$ s.t. $\V[P_\mathscr{L}] \leq c_1\V[P]$ and consequently another constant $\tilde{c}>0$ s.t. ${\V[P+P_\mathscr{L}]^{1/2} \leq \tilde{c}\V[P]^{1/2}}$.
	Combining equation \eqref{eq:th_hybrid_QMC1} with Cauchy-Schwarz and the embedding $L^4(G) \subset L^2(G)$ gives that
	\begin{align}
	\label{eq:th_hybrid_QMC_step2}
	\V[P-P_\mathscr{L}]^{1/2} \leq |D|^{1/4}\E[C(\omega)^4]^{1/4}\E[||u-u_\mathscr{L}||_{L^4(G)}^4]^{1/4}.
	\end{align}
	Now, owing to lemma 5.5 in \cite{CrociPHD}, we have that
	\begin{align}
	\label{eq:th_hybrid_QMC_step3}
	\E[||u-u_\mathscr{L}||_{L^4(G)}^4]^{1/4} \leq \bar{C}(s,D,d)|\square_H|^{\mu/d},
	\end{align}
	for some constant $\bar{C}$. Note that by construction $|\square_H| = 2^{-d\mathscr{L}}$. We can now pull together equations \eqref{eq:th_hybrid_QMC_step1}, \eqref{eq:th_hybrid_QMC_step2} and \eqref{eq:th_hybrid_QMC_step3} to obtain,
	\begin{align}
	\V[\hat{P}] \leq \hat{c}\frac{\V[P]}{N^q} + \frac{\bar{c}}{N}2^{-\mu \mathscr{L}},
	\end{align}
	where $\hat{c}=cc_1$ and
	\begin{align}
	\bar{c} = c\tilde{c}|D|^{1/4}\bar{C}(s,D,d)\E[C(\omega)^4]^{1/4}\V[P]^{1/2}.
	\end{align}
	and this concludes the proof.
\end{proof}

Theorem \ref{th:hybrid_QMC_conv} states that the statistical error introduced by approximating $\E[P]$ with our hybrid QMC technique can be split in two terms, where the former is the statistical error of a pure randomized QMC estimator and might converge faster than $O(N^{-1/2})$ and the second is a standard MC error correction term that exhibits the usual Monte Carlo rate, but decays geometrically as the Haar level increases. This splitting of the error directly relates to the splitting of white noise as the first term only depends on the truncation $\W_\mathscr{L}$. If $\W_\mathscr{L}$ well approximates $\W$, then we expect a pure QMC rate in the pre-asymptotic regime, while if the approximation is poor (small $\mathscr{L}$), then only a $O(N^{-1/2})$ rate can be expected.

\begin{remark}
	Another way of interpreting our hybrid approach is that we are splitting white noise into a smooth part $\W_\mathscr{L}$ and a rough part $\W_R$. QMC is effective at reducing the statistical error coming from the smooth part, but performs poorly when approximating the rough part and we are better off with directly using pseudo-random points. This aspect was experimentally investigated in \cite{Beentjes2018} and can be seen as another instance of the \emph{effective dimensionality} principle mentioned in section \ref{sec:background}.
\end{remark}

As we see in section \ref{sec:MLQMC_num_res}, even if the asymptotic rate is still $O(N^{-1/2})$, large gains are still obtained in practical experiments in the pre-asymptotic QMC-like regime, especially in a MLQMC setting where not that many samples are needed on the finest levels.

\end{document}